\documentclass[12pt,oneside,english]{amsart}

\usepackage{tcolorbox}
\usepackage{graphicx}
\usepackage{caption}

\usepackage{amsmath}
\usepackage{amssymb}
\usepackage{hyperref}
\usepackage{enumerate}
\usepackage{graphicx}
\usepackage{color}
\usepackage{xcolor}
\usepackage{mathtools}
\usepackage{float}
\usepackage{tikz}
\usetikzlibrary{arrows}
\usepackage{bm}
\usepackage[title]{appendix}
\usepackage{mathrsfs}
\DeclareMathAlphabet\euscript{U}{eus}{m}{n}
\usepackage{hyperref}
\hypersetup{
	colorlinks = true, 
	urlcolor = blue, 
	linkcolor = red, 
	citecolor = blue 
}

\usepackage{mathdots}
\usepackage{amsthm}
\usepackage{tikz,amsmath}
\usetikzlibrary{arrows}
\usetikzlibrary{calc}
\usepackage{tikz-cd}
\usetikzlibrary{decorations.pathreplacing}
\tikzset{%
	show curve controls/.style={
		postaction={
			decoration={
				show path construction,
				curveto code={
					\draw [blue,-] 
					(\tikzinputsegmentfirst) -- (\tikzinputsegmentsupporta)
					(\tikzinputsegmentlast) -- (\tikzinputsegmentsupportb);
					\fill [red, opacity=0.5] 
					(\tikzinputsegmentsupporta) circle [radius=.2ex]
					(\tikzinputsegmentsupportb) circle [radius=.2ex];
				}
			},
			decorate
		}
	},
	scc/.style={
	},
}
\usepackage{epsfig}
\usepackage[normalem]{ulem}
\usepackage[all,line,
]{xy} 
\CompileMatrices

\newcommand{\id}{\operatorname{id}}

\newcommand{\Per}{\operatorname{Per}}

\newtheorem{thm}{Theorem}[section]
\newtheorem{corollary}[thm]{Corollary}
\newenvironment{cor}{\begin{corollary}\rm}{\end{corollary}}
\newtheorem{lemma}[thm]{Lemma}

\newtheorem{prop}[thm]{Proposition}

\newtheorem{assumption}[thm]{Assumption}

\newtheorem{definition}[thm]{Definition}

\newtheorem{example}[thm]{Example}

\newtheorem{question}[thm]{Question}

\newtheorem{remark}[thm]{Remark}

\newtheorem{algorithm}[thm]{Algorithm}

\newtheorem{observation}[thm]{Observation}

\newtheorem{claim}[thm]{Claim}

\newtheorem{fact}[thm]{Fact}

\newenvironment{ex}{\begin{example}\rm}{\end{example}}
\newenvironment{re}{\begin{remark}\rm}{\end{remark}}

\newcommand{\vertiii}[1]{{\left\vert\kern-0.25ex\left\vert\kern-0.25ex\left\vert #1
    \right\vert\kern-0.25ex\right\vert\kern-0.25ex\right\vert}}

\title{On the future cover of a sofic shift}

\author{Klaus Thomsen}

\address{Department of Mathematics, Aarhus University, Ny Munkegade, 8000 Aarhus C, Denmark}

\email{matkt@math.au.dk}

\begin{document}

Address of the author:

Klaus Thomsen

Department of Mathematics, Aarhus University, Ny Munkegade, 8000 Aarhus C, Denmark

e-mail: matkt@math.au.dk

\begin{abstract} The paper contains a new proof of the theorem by Krieger which establishes the canonicity of the future cover of a sofic shift. In addition the paper describes a method to produce a new canonical cover from a given one, resulting in a canonical cover related to, but generally different from the future cover.
\end{abstract}
\maketitle\thispagestyle{empty}

\section{Introduction}
 
 A cover of a sofic shift $Y$ is a subshift of finite type $X$ together with a factor map $\pi : X \to Y$ from $X$ onto $Y$. Every sofic shift has a wealth of different covers, so the least one can do to simplify considerations regarding covers is to identify two covers $\pi :  X \to Y$ and $\pi' : X' \to Y$ of the same sofic shift $Y$ when they are isomorphic in the sense that there is a conjugacy $\chi : X \to X'$ such that $\pi' \circ \chi = \pi$. Recall that a subshift of finite type can be realized as the edge shift $X_G$ of a finite directed graph $G$ and that a labeling of the set of edges $E_G$ is a map $L_G : E_G \to A$, where $A$ is a finite set. If $Y = L_G(X_G)$, where $L_G : X_G \to Y \subseteq A^{\mathbb Z}$ is the map obtained by reading the labels of elements in $X_G$, we say that $(G,L_G)$ is a presentation of $Y$. Then any cover of $Y$ is isomorphic to one coming from a presentation of $Y$, see Theorem 3.2.1 in \cite{LM}. But even after having reduced covers to labeled graphs, every sofic shift has an uncanny host of different presentations. Fortunately the future cover of Wolfgang Krieger, \cite{Kr1}, and its associated labeled graph which we denote by $(\mathbb K(Y),L_{\mathbb K(Y)})$, stands out because it is canonical in the sense that a conjugacy $\psi : Y \to Z$ of sofic shifts lifts to a conjugacy $\psi_{\mathbb K} : X_{\mathbb K(Y)} \to X_{\mathbb K(Z)}$ such that
\begin{equation*}
\xymatrix{
 X_{\mathbb K(Y)} \ar[d]_-{L_{\mathbb K(Y)}} \ar[r]^-{\psi_{\mathbb K}} &  X_{\mathbb K(Z)} \ar[d]^-{L_{\mathbb K(Z)}} \\
Y \ar[r]_-{\psi}  & Z }
\end{equation*}
commutes. Thus the subshift of finite type $X_{\mathbb K(Y)}$ and the factor map $L_{\mathbb K(Y)}$ are both conjugacy invariants for $Y$. The same is of course also true for any other cover which is canonical in the same way. It must be emphasized that also components of a canonical cover that are not of maximal entropy as well as the wandering parts of the cover, which often arise even when $Y$ is mixing, also give rise to conjugacy invariants for $Y$. While any cover of a sofic shift is a window to its internal structure, a canonical cover guarantees that what you see is genuinely intrinsic. In this way canonical covers offer one of the few tools we have for the study of sofic shifts; in particular the structure that set them apart from subshifts of finite type. This observation is the main motivation for the work and the results that are presented below.

We start by giving a new proof of Krieger's theorem. While his proof is a tour de force in handling sliding block codes, the proof given here is more conceptual and leads to an abstract characterization of the future cover analogous to the abstract characterization of the canonical cover of an irreducible sofic shift known as the minimal right-resolving cover or Fischer cover, \cite{F}. Alternative proofs of parts of Krieger's theorem have been given before; first by Boyle, Kitchens and Marcus in \cite{BKM} and later by Nasu in \cite{N}. The proof presented here is also different from those. 

The main new result of the paper is a construction which produces a new canonical cover out of a given one. To appreciate this, recall that Krieger proved the uniqueness of the conjugacy $\psi_\mathbb K$ in the diagram above. Without this additional requirement it is obvious that a canonical cover is not unique up to isomorphism; for example the disjoint union of two copies of the future cover will again be canonical in the sense used so far, but not with unique lifts of conjugacies. In the following we shall therefore refer to a construction of covers as \emph{weakly canonical} when it gives lifts of conjugacies, but possibly lifts that are not unique, while a construction which gives rise to unique lifts of conjugacies, as the future cover does, will be called a \emph{strongly canonical} cover.   The main new result of the paper, summarized by Theorem \ref{31-10-25hx}, describes a method to produce a strongly canonical cover out of a given weakly canonical one. 

  Up to now only two strongly canonical covers for sofic shifts have been known; at least when we for simplicity ignore the doubling coming from the passage from right-resolving to left-resolving covers that can be achieved by 'time-reversal'. Besides the future cover it is the minimal right-resolving cover, and only for irreducible sofic shifts. With the construction in the present paper we have now at least four different ways to construct a strongly canonical cover of an irreducible sofic shift and at least two for a general sofic shift. We finish the paper by describing the results of the constructions for the even shift, showing that all four covers are mutually different in this case.

\bigskip

{\em Acknowledgement:} I thank Brian Marcus, Tom Meyerovitch and Chengyu Wu for the collaboration that produced \cite{MMTW}. I'm particularly grateful to Brian for renewing my interest in sofic shifts, and I thank him also for showing me his solution to Exercise 3.2.8 (e) in \cite{LM}. In addition, I thank a thorough but anonymous referee for finding a gap in an earlier version of the proof of Theorem \ref{31-10-25hx}. 
 
\section{Krieger's future cover and its canonicity}

\subsection{Setup and notation} 

Let $A$ be a finite set. The set $A^\mathbb Z$ of bi-infinite sequences 
$$
x = (x_i)_{i \in \mathbb Z} = \cdots x_{-3}x_{-2}x_{-1}x_0x_1x_2x_3 \cdots
$$
of elements from $A$ is a compact metric space; a metric can be defined for example as in Example 6.1.10 of \cite{LM}. The shift $\sigma$ is the homeomorphism of $A^\mathbb Z$ defined such that
$$
\sigma(x)_i := x_{i+1}.
$$
A closed subset $X \subseteq A^\mathbb Z$ is shift-invariant when $\sigma(X) = X$ and it is then called a \emph{subshift}. We shall mainly consider \emph{subshifts of finite type} (abbreviated to SFT) and \emph{sofic shifts}. See Definition 2.1.1 and Definition 3.1.3 in \cite{LM}.

The set $\mathbb W(X)$ of \emph{words} in a subshift $X$ consists of the finite strings $a_1a_2 \cdots a_n \in A^n$ with the property that $a_1a_2 \cdots a_n  = x_1x_2\cdots x_n$ for some element $x = (x_i)_{i \in \mathbb Z} \in X$. Given an element $x \in X$ and an integer $j \in \mathbb Z$ we denote by $x_{(-\infty,j]}$ the element 
$$
\cdots x_{k}x_{k+1} x_{k+2} \cdots x_j \in A^{(-\infty,j]},
$$
and by $X(-\infty,j]$ the set
$$
X(-\infty,j] := \left\{ x_{(-\infty,j]}: \ x \in X \right\}.
$$
The symbols $x_{[j,\infty)}$ and $X[j,\infty)$ have a similar meaning. The same goes for expressions like $x_{[i,j)}$ and $x_{(i,j]}$ when $i < j$ or $x_{[i,j]}$ when $i \leq j$. The set of $\sigma$-periodic points of a subshift $X$ will be denoted by $\Per(X)$.

 Let $G$ be a finite directed graph with edges (or arrows) $E_G$ and vertices $V_G$. For $e \in E_G$ let $s_G(e) \in V_G$ be the start vertex and $t_G(e)$ the terminal vertex of $e$. We shall always assume that $G$ does not contain sinks or sources; i.e. for every vertex $v$ there are arrows $e_1,e_2 \in E_G$ such that $t_G(e_1) = v = s_G(e_2)$. We extend the definition of $t_G$ to finite and left-infinite paths in $G$, and the definition of $s_G$ to finite and right-infinite paths in $G$, in the obvious way. We denote the edge shift of $G$ by $X_G$. That is, $X_G$ consists of the bi-infinite sequences 
 $$
 (e_i)_{i \in \mathbb Z} = \cdots e_{-3}e_{-2}e_{-1}e_0e_1e_2e_3 \cdots \ \in \ E_G^\mathbb Z
 $$ 
 of edges in $G$ such that $s_G(e_{i+1}) = t_G(e_i)$ for all $i \in \mathbb Z$, and the words in $\mathbb W(X_G)$ are the finite paths in $G$. We shall sometimes refer to the elements of $X_G$ as  \emph{rays} in $G$. The set $X_G$ is an SFT and every SFT is conjugate to $X_G$ for some graph $G$ by Theorem 2.3.2 in \cite{LM}.
 
Let $A$ be a finite set (the alphabet) and $L_G : E_G \to A$ a map considered as a labelling of the edges in $G$. We extend $L_G$ to finite or infinite paths in the natural way. In particular, when $x=(e_i)_{i \in \mathbb Z} \in X_G$,
\begin{align*}
&L_G(x) := \left(L_G(e_i)\right)_{i \in \mathbb Z} = \cdots L_G(e_{-3}) L_G(e_{-2}) L_G(e_{-1}) L_G(e_0)L_G(e_1)L_G(e_2) L_G(e_3) \cdots \\
&\in A^\mathbb Z .
\end{align*}
Then $(G,L_G)$ is a \emph{labeled graph}, $Y:= L_G(X_G) \subseteq A^\mathbb Z$ is a sofic shift and $(G,L_G)$ is a \emph{presentation} of $Y$, cf. \S 3.1 in \cite{LM}.
 
The first lemma deals with factor codes, right-resolving graphs and right-closing factor codes. We refer to Definition 3.3.1 and Definition 8.1.8 in \cite{LM}, respectively. The book \cite{LM} is the bible of symbolic dynamics and we shall sometimes use notions and results from \cite{LM} without explicit reference. 

\begin{lemma}\label{05-05-25x} Let $X$ be an SFT and $\pi : X \to Y$ a right-closing factor code. There is a right-resolving labeled graph $(G,L_G)$ and a conjugacy $\psi : X \to X_G$ such that

\begin{equation*}
\begin{xymatrix}{
 X \ar[rr]^\psi \ar[dr]_\pi & & X_G \ar[dl]^{L_G} \\
 & Y &}
\end{xymatrix}
\end{equation*}
commutes.
\end{lemma}
\begin{proof} Recode, first using the proof of Theorem 3.2.1 in \cite{LM} to get a labeled graph $(H,L_H)$ and a conjugacy $\psi_0 : X \to X_H$ such that

\begin{equation*}
\begin{xymatrix}{
 X \ar[rr]^{\psi_0} \ar[dr]_\pi & & X_H \ar[dl]^{L_H} \\
 & Y &}
\end{xymatrix}
\end{equation*}
commutes. Note that $L_H : X_H \to Y$ is a right-closing factor code since $\pi$ is. It follows therefore from Proposition 8.1.9 in \cite{LM} that $(H,L_H)$ is right-closing with delay as defined in Definition 5.1.4 of \cite{LM}.  Consequently Proposition 5.1.11 in \cite{LM} gives us a right-resolving labeled graph $(G,L_G)$ and a conjugacy $\psi_1 : X_H \to X_G$ such that
\begin{equation*}
\begin{xymatrix}{
 X_H \ar[rr]^{\psi_1} \ar[dr]_{L_H} & & X_G \ar[dl]^{L_G} \\
 & Y &}
\end{xymatrix}
\end{equation*}
commutes. Set $\psi :=  \psi_1 \circ \psi_0$.

\end{proof}

\subsection{Regular vertices and rays}\label{rays} 

Let $(H,L_H)$ be a labeled graph and $Y = L_H(X_H)$ the sofic shift it presents. A vertex $v \in V_H$ is \emph{regular} when there is an element $z \in X_H$ such that $t_H(z_{(-\infty,-1]}) = v$ and 
\begin{equation}\label{11-09-25}
 \left\{ L_H(x) : \ x \in X_H{[0,\infty)}, \ s_H(x) = v \right\} = \left\{y \in Y[0,\infty) : \ L_H(z_{(-\infty,-1]})y \in Y \right\} .
\end{equation}
For any element $z \in X_H$ with $t_H(z_{(-\infty,-1]}) = v$, the inclusion 
$$
\left\{ L_H(x) : \ x \in X_H{[0,\infty)}, \ s_H(x) = v \right\} \subseteq \left\{y \in Y[0,\infty) : \ L_H(z_{(-\infty,-1]})y \in Y \right\}
$$ 
always holds, so the real condition in \eqref{11-09-25} is that the opposite inclusion,
$$
\left\{ L_H(x) : \ x \in X_H{[0,\infty)}, \ s_H(x) = v \right\} \supseteq \left\{y \in Y[0,\infty) : \ L_H(z_{(-\infty,-1]})y \in Y \right\},
$$ 
also holds.

In order to work with regularity of vertices we introduce the following notation. Given an element $y \in Y$, set
\begin{equation}\label{23-09-25a}
F(y) := \left\{w \in Y[0,\infty): \ y_{(-\infty,-1]}w \in Y \right\} ,
\end{equation}
and given a vertex $v \in V_H$, set
$$
f_H(v) : = \left\{L_H(x): \ x \in X_H[0,\infty), \ s_H(x) = v \right\} .
$$
Then $v \in V_H$ is regular when there is an element $z \in X_H$ such that $t_H(z_{(-\infty,-1]}) =v$ and 
\begin{equation}\label{15-09-12}
f_H(v) = F(L_H(z)) .
\end{equation}
The set $f_H(v)$ will be called the \emph{follower set} of $v$ and $F(y)$ the \emph{follower set} of $y$.

\begin{lemma}\label{31-08-25x} Assume that $(H,L_H)$ is right-resolving. Let $\gamma$ be a finite path in $H$, and assume that $s_H(\gamma)$ is regular. Then $t_H(\gamma)$ is regular.
\end{lemma}
\begin{proof} Let $z \in X_H$ be an element such that $t_H(z_{(-\infty,-1]}) = s_H(\gamma)$ and
\begin{equation}\label{03-11-25}
 \left\{ L_H(\mu) : \ \mu \in X_H{[0,\infty)}, \ s_H(\mu) =s_H(\gamma)   \right\} = \left\{y \in Y[0,\infty) : \ L_H(z_{(-\infty,-1]})y \in Y \right\} .
\end{equation}
 Set $\gamma' := z_{(-\infty,-1]}\gamma$ and note that $\gamma' = z'_{(-\infty,-1]}$ for some $z' \in X_H$ with $t_H(z'_{(-\infty,-1]}) = t_H(\gamma)$. If $y\in Y[0,\infty)$ and $L_H(z'_{(-\infty,-1]})y = L_H( z_{(-\infty,-1]})L_H(\gamma)y \in Y$, then \eqref{03-11-25} implies that $L_H(\gamma)y = L_H(\nu)$ for some $\nu \in X_H[0,\infty)$ with $s_H(\nu) = s_H(\gamma)$. Since $(H,L_H)$ is right-resolving, $\nu = \gamma \nu'$ for some $\nu' \in X_H[0,\infty)$ with $s_H(\nu') = t_H(\gamma)$. So $y =L_H(\nu')$.
\end{proof}

A labeled graph will be called \emph{regular} when all vertices are regular. Adapting Definition 3.3.7 in \cite{LM} we say that a labeled graph $(H,L_H)$ is \emph{follower-separated} when $f_H(v) = f_H(w) \Rightarrow v =w$.

\begin{ex}\label{03-10-25} The following labeled graph $(G,L_G)$ is right-resolving and follower-separated. But it is not regular; the vertex a is not regular. Indeed, for any $z \in X_G$ with $t_G(z_{(-\infty,-1]}) = a$ we have that $31^{\infty} \in F(L_G(z)) \backslash f_G(a)$.

\begin{figure}[H]
\begin{equation*}
\begin{tikzpicture}[node distance={25mm}, thick, main/.style = {draw, circle}] 
\node[main] (1) {$a$}; 
\node[main] (2) [right of=1] {$b$}; 
\node[main] (3) [right of=2] {$c$};
\draw[->] (1) to [out=210,in=310,looseness=7] node[below ,pos=0.4] {$1$} (1);
\draw[->] (1) to [out=0,in=180,looseness=1] node[above ,pos=0.4] {$2$} (2);
\draw[->] (2) to [out=110,in=70,looseness=15] node[above ,pos=0.4] {$1$} (2);
\draw[->] (3) to [out=210,in=270,looseness=10] node[below ,pos=0.4] {$1$} (3);
\draw[->] (3) to [out=290,in=330,looseness=15] node[below ,pos=0.4] {$4$} (3);
\draw[->] (3) to [out=180,in=0,looseness=1] node[above ,pos=0.4] {$3$} (2);
\end{tikzpicture} 
\end{equation*}
\caption{A right-resolving follower-separated labeled graph which is not regular.}
\label{04-09-25gxx}

\end{figure}
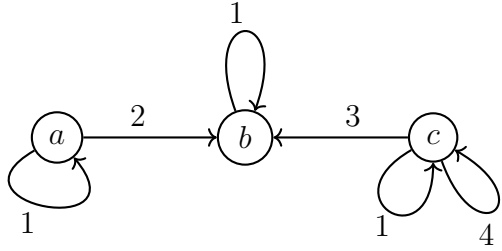
\end{ex}

\begin{re}\label{26-09-25x} 
Every synchronizing graph in the sense of Jonoska, cf. \S 4 in \cite{Jo}, is regular. By Lemma 3.3.15, Proposition 3.3.16 and Corollary 3.3.19 in \cite{LM} the minimal right-resolving presentation of an irreducible sofic shift is synchronizing in the sense of Jonoska and therefore also regular. The labeled graph in Figure \ref{11-09-25cx} depicts the future cover of the even shift and it is regular, but it is not synchronizing. To give an irreducible example consider the following labeled graph which is strongly connected, right-resolving, and presents the even shift.
\begin{figure}[H]
\begin{equation*}
\begin{tikzpicture}[node distance={25mm}, thick, main/.style = {draw, circle}] 
\node[main] (1) {$a$}; 
\node[main] (2) [right of=1] {$b$}; 
\node[main] (3) [right of=2] {$c$};
\node[main] (4) [right of=3] {$d$};
\draw[->] (1) to [out=340,in=210,looseness=1.0] node[below ,pos=0.5] {$0$} (2);
\draw[->] (2) to [out=110,in=60,looseness=1.0] node[above,pos=0.5] {$0$} (1);
\draw[->] (2) to [out=60,in=110,looseness=1.0] node[above,pos=0.5] {$1$} (3);
\draw[->] (3) to [out=200,in=340,looseness=1.0] node[below,pos=0.5] {$1$} (2);
\draw[->] (3) to [out=60,in=110,looseness=1.0] node[above,pos=0.5] {$0$} (4);
\draw[->] (4) to [out=200,in=340,looseness=1.0] node[below,pos=0.5] {$0$} (3);
\end{tikzpicture} 
\end{equation*}
\caption{A strongly connected labeled graph which is right-resolving and regular, but not synchronizing.}
\label{26-09-25}
\end{figure}
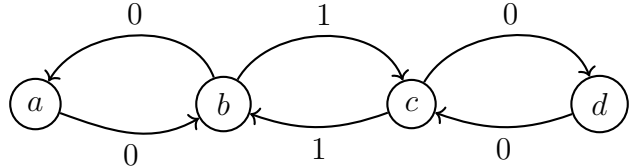

\end{re}

Let $x \in X_H$. We are interested in the condition that 
\begin{equation}\label{02-09-25ax}
\left\{ L_H(\mu) : \ \mu \in X_H[0,\infty), \ s_H(\mu)  = 
t_H(x_{(-\infty, k]})\right\} = \left\{ y \in Y[0,\infty): \ L_H(x_{(-\infty,k]})y \in Y \right\}, 
\end{equation}
or 
\begin{equation}\label{23-09-25}
f_H(t_H(x_{(-\infty,k]})) = F(\sigma^{k+1}(L_H(x))),
\end{equation}
for all $k \in \mathbb Z$. Let $X$ be a subshift and $d$ a metric for the topology of $X$. For $x \in X$, set 
\begin{align*}
&\mathbb U(x) := \left\{ z \in X: \ \lim_{i \to - \infty} d\left(\sigma^i(z),\sigma^i(x)\right) = 0\right\} \\
&= \left\{z \in X: \ z_i = x_i, \ i \leq N, \ \text{for some} \ N \in \mathbb Z \right\} .
\end{align*} 
This set is sometimes called 'the unstable manifold' of $x$.

\begin{lemma}\label{02-09-25x}  Assume that $(H,L_H)$ is right-resolving. Let $x \in X_H$. Then $L_H : \mathbb U(x) \to \mathbb U(L_H(x))$ is surjective if and only if \eqref{02-09-25ax} holds for all $k \in \mathbb Z$.

\end{lemma}
\begin{proof} 

 Assume first that \eqref{02-09-25ax} holds for all $k \in \mathbb Z$ and consider an element $y \in \mathbb U(L_H(x))$. Then $y_{(-\infty,L]} = L_H(x_{(-\infty,L]})$ for some $L \in \mathbb Z$. Thanks to \eqref{02-09-25ax} this implies that $y_{[L+1,\infty)} \in f_H(t_H(x_{(-\infty,L]}))$, and hence also that there is an element $z \in X_H$ such that $z_{(-\infty,L]} = x_{(-\infty,L]}$ and $L_H(z_{[L+1,\infty)}) = y_{[L+1,\infty)}$. Then $z \in \mathbb U(x)$ and $L_H(z) = y$.

Assume then that $L_H : \mathbb U(x) \to \mathbb U(L_H(x))$ is surjective, and consider $k \in \mathbb Z$ and an element $y\in Y[0,\infty)$ such that $L_H(x_{(-\infty,k]})y \in Y$. Then $L_H(x_{(-\infty,k]})y \in \mathbb U(L_H(x))$ and there is therefore an element $z \in \mathbb U(x)$ such that $L_H(z) =L_H(x_{(-\infty,k]})y$. Choose $L < k$ such that $z_{(-\infty, L]} = x_{(-\infty, L]}$. Since $(H,L_H)$ is right-resolving, it follows that $z_{(-\infty, k]} = x_{(-\infty, k]}$. Hence 
$$
y = L_H(z_{[k+1,\infty)}) \in f_H(t_H(x_{(-\infty,k]})).
$$

\end{proof}


An element $x \in X_H$ will be called \emph{regular} when $L_H : \mathbb U(x) \to \mathbb U(L_H(x))$ is surjective. By Lemma \ref{02-09-25x} the following clearly holds.

\begin{lemma}\label{12-09-25b} Assume that $(H,L_H)$ is right-resolving and that $x\in X_H$ is regular. Then $s_H(x_k)\in V_H$ is regular for all $k \in \mathbb Z$.
\end{lemma}

We denote in the following the set of regular rays in $X_H$ by $\mathcal R(L_H)$.

\begin{ex}\label{03-10-25a}
In the labeled graph $H$ depicted in Figure \ref{11-09-25cx} the vertex $c$ is regular and hence all the vertices are regular by Lemma \ref{31-08-25x}. However, $X_H$ contains non-regular rays; the two elements of $X_H$ of period 2 are not regular. Indeed, if $x$ is one of them we have either $1^{\infty} \in  F(L_H(x)) \backslash f_H(t_H(x_{(-\infty,-1]}))$ or $1^{\infty} \in  F(\sigma(L_H(x))) \backslash f_H(t_H(x_{(-\infty,0]}))$.

\begin{figure}[H]
\begin{equation*}
\begin{tikzpicture}[node distance={25mm}, thick, main/.style = {draw, circle}] 
\node[main] (1) {$a$}; 
\node[main] (2) [right of=1] {$b$}; 
\node[main] (3) [below of=1] {$c$};
\draw[->] (1) to [out=80,in=110,looseness=1.0] node[above ,pos=0.5] {$0$} (2); 
\draw[->] (2) to [out=260,in=310,looseness=1.0]  node[below ,pos=0.5] {$0$}(1); 
\draw[->] (1) to [out=160,in=220,looseness=15] node[above ,pos=0.4] {$1$} (1);
\draw[->] (3) to [out=90,in=270,looseness=1.0]  node[right ,pos=0.5] {$1$} (1);
\draw[->] (3) to [out=10,in=290,looseness=10]  node[below ,pos=0.5] {$0$} (3);
\end{tikzpicture} 
\end{equation*}
\caption{A regular labeled graph with non-regular rays.}
\label{11-09-25cx}
\end{figure}
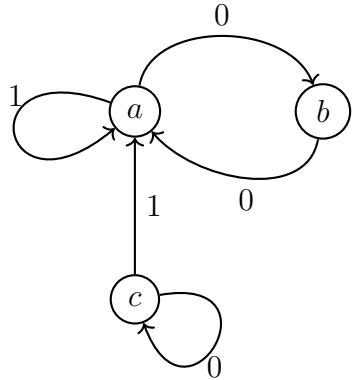
\end{ex}


\begin{lemma}\label{02-09-25bx} Assume that $(X_H,L_H)$ is right-resolving and follower-separated. Then $L_H$ is injective on $\mathcal R(L_H)$.
\end{lemma}
\begin{proof} If $x,y \in \mathcal R(L_H)$ and $L_H(x) = L_H(y)$, it follows first from \eqref{23-09-25} that $t_H(x_k) = t_H(y_k)$ for all $k \in \mathbb Z$ because we assume that $(X_H,L_H)$ is follower-separated. Since $(H,L_H)$ is right-resolving it follows that $x=y$.
\end{proof}

\subsection{The future cover and Krieger's theorem}

Let $A$ be a finite set and 
$Y \subseteq A^\mathbb Z$ a sofic subshift. To define the future cover of $Y$ we use the following notation. When $D \subseteq A^\mathbb N$ is a subset of $A^{\mathbb N}$ and $a \in A$, we let $\frac{D}{a}$ denote the set
$$
\frac{D}{a} := \left\{x_1x_2x_3 \cdots \in A^\mathbb N : \ x_0x_1x_2\cdots \in D, \ x_0 =a \right\} = \left\{x \in A^\mathbb N: \ ax \in D\right\};
$$
i.e. $\frac{D}{a}$ is the set of elements of $A^{\mathbb N}$ that are obtained from the elements of $D$ that start with $a$ by deleting the first coordinate. More generally, when $w \in A^m$, we set
$$
\frac{D}{w} = \left\{x \in A^{\mathbb N} : \ wx \in D \right\}.
$$

Let $y \in Y$. Recall from \eqref{23-09-25a} that
$$
F(y)  = \left\{ z_{[0,\infty)}: \ z \in Y, \ z_j = y_j, \ j \leq -1 \right\} \subseteq A^{\mathbb N}. 
$$
Set
$$
V_{\mathbb K(Y)} := \left\{F(y): \ y \in Y\right\} .
$$
This is a finite set because $Y$ is sofic. $V_{\mathbb K(Y)}$ is the set of vertices in a labeled graph $(\mathbb K(Y),L_{\mathbb K(Y)})$ where there is a labeled arrow
$$
F(y) \overset{a}{\to} F(z)
$$
from $F(y)$ to $F(z)$ in $\mathbb K(Y)$ precisely when 
$F(z) = \frac{F(y)}{a}$.
The labeled graph $(\mathbb K(Y),L_{\mathbb K(Y)})$ is then clearly right-resolving. Since 
$$
F(\sigma^{i+1}(y)) = \frac{F(\sigma^{i}(y))}{y_i},
$$
there is an element $\alpha_Y(y) \in X_{\mathbb K(Y)}$ determined by the two conditions:
\begin{itemize}
\item[a)] $s_{\mathbb K(Y)}(\alpha_Y(y)_i) = F(\sigma^i(y)), \ i \in \mathbb Z$, and
\item[b)] $L_{\mathbb K(Y)}(\alpha_Y(y)) = y$.
\end{itemize}
The map $\alpha_Y : Y \to X_{\mathbb K(Y)}$ commutes with the shift, i.e. $\alpha_Y \circ \sigma = \sigma \circ \alpha_Y$, but it is not continuous.

\begin{lemma}\label{04-09-25f} $L_{\mathbb K(Y)}(X_{\mathbb K(Y)}) =Y$.
\end{lemma}
\begin{proof} Let $x = (x_i)_{i \in \mathbb Z} \in X_{\mathbb K(Y)}$ and consider an $N \in \mathbb N$. By definition $s_{\mathbb K(Y)}(x_{-N}) = F(y)$ for some $y \in Y$ and
$$
L_{\mathbb K(Y)}(x_{-N})L_{\mathbb K(Y)}(x_{-N+1}) L_{\mathbb K(Y)}(x_{-N+2}) \cdots  \ \ \in F(y) .
$$
In particular,
$$
L_{\mathbb K(Y)}(x_{-N})L_{\mathbb K(Y)}(x_{-N+1}) \cdots L_{\mathbb K(Y)}(x_{N}) \in \mathbb W(Y)
$$
for all $N \in \mathbb N$, implying that $L_{\mathbb K(Y)}(X_{\mathbb K(Y)}) \subseteq Y$. The reversed inclusion, $Y \subseteq L_{\mathbb K(Y)}(X_{\mathbb K(Y)})$, follows from the observation b) above.
\end{proof}

Thus $(\mathbb K(Y),L_{\mathbb K(Y)})$ is a presentation of $Y$, and it is known as \emph{the future cover} of $Y$, cf. Exercise 3.2.8 on page 75 of \cite{LM}. \footnote{Krieger introduced the word 'future' in \cite{Kr1}, and called an equivalent version of $(\mathbb K(Y),L_{\mathbb K(Y)})$ 'the future state chain'. The terminology "future cover" was used in \cite{BKM}, but only for a component of $(\mathbb K(Y),L_{\mathbb K(Y)})$, and only for transitive sofic shifts. In Exercise 3.2.8 of \cite{LM} the quotation marks were deleted and $(\mathbb K(Y),L_{\mathbb K(Y)})$ was called the future cover for general sofic shifts. In some papers on sofic shifts the future cover is called the Krieger cover and the minimal right-resolving cover is called the Fischer cover.  }

We can now formulate Krieger's theorem on the canonicity of the future cover. In our notation Theorem 2.14 in \cite{Kr1} reads as follows: 

\begin{lemma}\label{03-02-25ax} (Theorem 2.14 of \cite{Kr1})  Let $Y$ and $Z$ be sofic shifts and $\psi : Y \to Z$ a conjugacy. There is a unique conjugacy $\psi_{\mathbb K} : X_{\mathbb K(Y)} \to X_{\mathbb K(Z)}$ such that
\begin{itemize}
\item[(a)] $\psi \circ L_{\mathbb K(Y)} = L_{\mathbb K(Z)} \circ \psi_{\mathbb K}$, and
\item[(b)] $\psi_{\mathbb K} \circ \alpha_Y = \alpha_{Z} \circ \psi$.
\end{itemize}
\end{lemma}

When $Y$ (and $Z$) are transitive, Krieger was able to partly dispense with condition (b) in Lemma \ref{03-02-25ax}, cf. Corollary 2.16 in \cite{Kr1}. And in Theorem 2.2 of \cite{Kr2} he removed condition (b) entirely by showing that it follows from (a). In this way he proved

\begin{thm}\label{Krieger} (W. Krieger, 1984, 1987)
Let $Y$ and $Z$ be sofic shifts and $\psi : Y \to Z$ a conjugacy. There is a unique conjugacy $\psi_{\mathbb K} : X_{\mathbb K(Y)} \to X_{\mathbb K(Z)}$ such that
\begin{equation*}\label{24-11-24}
\xymatrix{
 X_{\mathbb K(Y)} \ar[d]_-{L_{\mathbb K(Y)}} \ar[r]^-{\psi_{\mathbb K}} &  X_{\mathbb K(Z)} \ar[d]^-{L_{\mathbb K(Z)}} \\
Y \ar[r]_-{\psi}  & Z }
\end{equation*}
commutes. This conjugacy $\psi_{\mathbb K}$ has the property that also
\begin{equation}\label{24-11-24x}
\xymatrix{
 X_{\mathbb K(Y)} \ar[r]^-{\psi_{\mathbb K}} &  X_{\mathbb K(Z)} \\
Y \ar[r]_-\psi \ar[u]^-{\alpha_Y} & Z \ar[u]_-{\alpha_{Z}} }
\end{equation}
commutes.
\end{thm}

Note that \eqref{24-11-24x} says that the right-inverse $\alpha_Y$ of $L_{\mathbb K(Y)}$ is canonical in the same sense as $(\mathbb K(Y),L_{\mathbb K(Y)})$. It can therefore also be used to get conjugacy invariants for $Y$.

Before we can give the proof of Theorem \ref{Krieger} we need some preparations. First of all we need the relation
\begin{equation}\label{11-09-25h}
f_{\mathbb K(Y)}(F(y)) = F(y)
\end{equation}
in which $F(y)$ is first considered as a vertex in $\mathbb K(Y)$ and next as a subset of $A^{\mathbb N}$. In more detail,

\begin{lemma}\label{02-10-24fx} $(\mathbb K(Y),L_{\mathbb K(Y)})$ is follower-separated. In fact, for all $y \in Y$,
$$
\left\{L_{\mathbb K(Y)}(w) :  \ w \in X_{\mathbb K(Y)}[0,\infty), \ s_{\mathbb K(Y)}(w) =  F(y)\right\} = F(y).
$$
\end{lemma}
\begin{proof} This follows directly from the definition of $\mathbb K(Y)$.
\end{proof}

\begin{lemma}\label{11-09-25g} $\alpha_Y(y)$ is regular in $X_{\mathbb K(Y)}$ for all $y \in Y$, i.e.
$\alpha_Y(Y) \subseteq \mathcal R(L_{\mathbb K(Y)})$.
\end{lemma}
\begin{proof} Since $L_{\mathbb K(Y)}(\alpha_Y(y)) = y$ we must check that
$$
f_{\mathbb K(Y)}\left(t_{\mathbb K(Y)}(\alpha_Y(y)_{(-\infty,k]})\right) = F(\sigma^{k+1}(y))
$$
for $k \in \mathbb Z$. This is a direct check, using \eqref{11-09-25h}: 
$$f_{\mathbb K(Y)}\left(t_{\mathbb K(Y)}(\alpha_Y(y)_{(-\infty,k]})\right) = f_{\mathbb K(Y)}\left(s_{\mathbb K(Y)}(\alpha_Y(y)_{k+1})\right)  = f_{\mathbb K(Y)}(F(\sigma^{k+1}(y))) = F(\sigma^{k+1}(y)).
$$
\end{proof}

\begin{lemma}\label{04-09-25h} $(\mathbb K(Y),L_{\mathbb K(Y)})$ is regular.
\end{lemma}
\begin{proof}  Let $y \in Y$. To see that $F(y)$ is a regular vertex in $\mathbb K(Y)$, recall that $(\mathbb K(Y),L_{\mathbb K(Y)})$ is right-resolving. Since $F(y) = s_{\mathbb K(Y)}(\alpha_Y(y)_0)$ the desired conclusion follows therefore by combining Lemma \ref{11-09-25g} with Lemma \ref{12-09-25b}.
\end{proof}

\begin{lemma}\label{12-09-25j} Let $\gamma : Y \to X_{\mathbb K(Y)}$ be a map such that $\gamma(Y) \subseteq \mathcal R(L_{\mathbb K(Y)})$ and $L_{\mathbb K(Y)}\circ \gamma =\id_Y$.  Then $\gamma = \alpha_Y$.   
\end{lemma}
\begin{proof} By definition and by Lemma \ref{11-09-25g} $\alpha_Y$ has the same properties as $\gamma$. Since $L_{\mathbb K(Y)}(\alpha_Y(y)) = y = L_{\mathbb K(Y)}(\gamma(y))$ for all $y \in Y$, and since $(\mathbb K(Y),L_{\mathbb K(Y)})$ is follower-separated by Lemma \ref{02-10-24fx}, it follows from Lemma \ref{02-09-25bx} that $\gamma = \alpha_Y$.
\end{proof}

Given an element $w \in A^m$ we denote in the following by $w^\infty$ the $m$-periodic element $w^\infty$ in $\Per(A^\mathbb Z)$ with the property that $(w^\infty)_{[0,m-1]} = w$.

\begin{lemma}\label{11-09-25ix} Let $d$ be a metric for the topology of $X_{\mathbb K(Y)}$, let $x \in X_{\mathbb K(Y)}$ and let $N\in \mathbb Z$. There is a periodic point $p \in \Per(Y)$ and an element $x' \in X_{\mathbb K(Y)}$ such that $x'_i = x_i, \ i \geq N$, and $\lim_{i \to -\infty} d(\sigma^i(x'),\sigma^i(\alpha_Y(p))) = 0$. In particular, $x' \in \mathcal R(L_{\mathbb K(Y)})$.
\end{lemma}
\begin{proof} Choose $y \in Y$ such that $F(y) = s_{\mathbb K(Y)}(x_N)$. By the pigeon hole principle there are integers 
$\left\{n_i\right\}_{i \in \mathbb N}$ such that 
$$\cdots < n_3 < n_2 <n_1 < n_0 < N
$$ 
and $s_{\mathbb K(Y)}(\alpha_Y(y)_{n_j}) = s_{\mathbb K(Y)}(\alpha_Y(y)_{n_0})$ for all $j$. For convenience we arrange that $n_0< -1$. Let $q_j := L_{\mathbb K(Y)}(\alpha_Y(y)_{[n_j,n_0-1]}) = y_{[n_j,n_0-1]}$ and note that $p^j := q_j^{\infty} \in Y$. We claim that $F(\sigma^{n_0}(y)) \subseteq F(p^j)$ for all $j \in \mathbb N$.
To see this, let $z \in F(\sigma^{n_0}(y))$.
Since $F(\sigma^{n_0}(y)) = s_{\mathbb K(Y)}(\alpha_Y(y)_{n_j}) =s_{\mathbb K(Y)}(\alpha_Y(y)_{n_0})$ and $L_{\mathbb K(Y)}(\alpha_Y(y)_{[n_j,n_0-1]}) = q_j$, it follows that $F(\sigma^{n_0}(y)) = \frac{ F(\sigma^{n_0}(y))}{q_j}$, and hence also that
$$
F(\sigma^{n_0}(y)) = \frac{ F(\sigma^{n_0}(y))}{(q_j)^k}
$$
for all $k \in \mathbb N$. Hence
\begin{align*}
&F(\sigma^{n_0}(y)) = \left\{w \in F(\sigma^{n_0}(y)) : \ (q_j)^k w \in F(\sigma^{n_0}(y))\  \forall k \in \mathbb N \right\} ,
\end{align*} 
which implies that $p^j_{(-\infty,-1]}z \in  Y$, i.e. $z \in F(p^j)$.
 We claim next that $F(\sigma^{n_0}(y)) = F(p^l)$ for some $l \in \mathbb N$. By passing to a subsequence we may assume for this that $F(p^j) = F(p^1)$ for all $j$. If $F(p^1) \backslash F(\sigma^{n_0}(y)) \neq \emptyset$, there is an element $z \in Y[0,\infty)$ such that $y_{[n_j,n_0-1]}z \in Y[0,\infty)$ for all $j > 1$, but $y_{(-\infty,n_0-1]}z \notin Y$, which is impossible. This proves the claim. It follows that $\alpha_Y(y)_{[n_0,-1]}$ is a path in $\mathbb K(Y)$, labeled by $y_{[n_0,-1]}$, which starts at $ F(\sigma^{n_0}(y)) =F(p^l) = s_{\mathbb K(Y)}(\alpha_Y(p^l)_0)$ and ends at $s_{\mathbb K(Y)}(\alpha_Y(y)_0) = F(y)= s_{\mathbb K(Y)}(x_N)$. Set $q:= \sigma^{-n_0}(p^l)$ and 
 $$
 x'':= \alpha_Y(q)_{(-\infty,n_0-1]}\alpha_Y(y)_{[n_0,-1]}x_{[N,\infty)} \in X_{\mathbb K(Y)}.
 $$ 
  Then $p := \sigma^{-N}(q)$ and $x' := \sigma^{-N}(x'')$ have the desired properties, and $x'$ is regular because $\mathbb U(x') = \mathbb U(\alpha_Y(p))$ and $\alpha_Y(p) \in \mathcal R(L_{\mathbb K(Y)})$ by Lemma \ref{11-09-25g}.

\end{proof}

\begin{corollary}\label{12-09-25} For every element $x \in X_{\mathbb K(Y)}$ and every $N \in \mathbb Z$ there is a regular ray $x' \in \mathcal R(L_{\mathbb K(Y)})$ such that $x'_i = x_i$ for $i \geq N$.
\end{corollary}
\begin{proof} This follows from Lemma \ref{11-09-25ix}.
\end{proof}

In the following we shall work with labeled-graph homomorphisms, cf. Definition 3.1.2 of \cite{LM}. Let $(G,L_G)$ and $(H,L_H)$ be labeled graphs. A \emph{labeled-graph homomorphism} $\theta : (G,L_G) \to (H,L_H)$ is a graph homomorphism $\theta: G \to H$ such that $L_H \circ \theta(e)= L_G(e)$ for all $e \in E_G$. The sliding block code $X_G \to X_H$ induced by $\theta$, which we also denote by $\theta$, has the property that $L_H \circ \theta = L_G$ on $X_G$.

\begin{lemma}\label{02-09-25dx} Assume that $(G,L_G)$ is a right-resolving and regular presentation of $Y$. It follows that there is a labeled-graph homomorphism $\theta : (G,L_G) \to (\mathbb K(Y), L_{\mathbb K(Y)})$ such that
$$
\theta\left(\mathcal R(L_G)\right) \subseteq \mathcal R(L_{\mathbb K(Y)}).
$$
If $(G,L_G)$ also is follower-separated, $\theta : G \to \mathbb K(Y)$ is injective.
\end{lemma}
\begin{proof} Let $v \in V_G$. Since $v$ is regular there is an $x \in X_G$ such that $t_G(x_{(-\infty,-1]}) = v$ and 
$f_G(v)  = F(L_G(x))$.
Define $\theta(v) \in V_{\mathbb K(Y)}$ such that 
$$
\theta(v) := F(L_G(x)) .
$$
This is well-defined: If $x'$ has the same properties as $x$, it follows that
$F(L_G(x')) = f_G(v) = F(L_G(x))$.

Let $e \in E_G$. Choose $x \in X_G$ such that $t_G(x_{(-\infty,-1]}) = s_G(e)$ and 
$f_G(s_G(e)) = F(L_G(x))$.
Set $x'_i = x_{i+1}, \ i \leq -2$, $x'_{-1} =e$ and let $x'_i, i \geq 0$, be arbitrary subject to the condition that $x' \in X_G$. Then $t_G(x'_{(-\infty, -1]}) = t_G(e)$ and by using that $(G,L_G)$ is right-resolving it follows that
\begin{align*}
& f_G(t_G(e)) = \frac{f_G(s_G(e))}{L_G(e)}
 = \frac{\left\{y \in Y[0,\infty): \ L_G(x_{(-\infty,-1]})y \in Y \right\}}{L_G(e)} \\
 &= \left\{y \in Y[0,\infty): \ L_G(x'_{(-\infty,-1]})y \in Y \right\}.
\end{align*}
Hence $\theta(t_G(e)) = F(L_G(x'))$ and  $F(L_G(x')) = \frac{F(L_G(x))}{L_G(e)}$. There is therefore an edge $e'$ in $\mathbb K(Y)$ from $\theta(s_G(e)) =F(L_G(x)) $ to $\theta(t_G(e)) = F(L_G(x'))$ with label $L_{\mathbb K(Y)}(e') =L_G(e)$. Since $(\mathbb K(Y), L_{\mathbb K(Y)})$ is right-resolving this arrow is unique and we set $\theta(e):= e'$. Then $\theta : G \to \mathbb K(Y)$ is a labeled-graph morphism.

When $x \in \mathcal R(L_G)$, it follows from the definition of $\theta$ and the defining relation \eqref{02-09-25ax} for regular rays that $\theta(x) \in X_{\mathbb K(Y)}$ is the path
$$
\cdots  \overset{L_G(x_{k-1})}{\to}F(L_G(\sigma^k(x))) \overset{L_G(x_k)}{\to} F(L_G(\sigma^{k+1}(x))) \cdots .
$$ 
Thus $\theta(x) = \alpha_Y(L_G(x))$ and hence $\theta(x) \in \mathcal R(L_{\mathbb K(Y)})$ by Lemma \ref{11-09-25g}.


Finally, assume that $(G,L_G)$ is follower-separated and that $v,w \in V_G$ are such that $\theta(v) =\theta(w)$. There are elements $x,x' \in X_G$ such that $t_G(x_{(-\infty,-1]}) = v$ and $t_G(x'_{(-\infty,-1]}) = w$, $f_{G}(v) = F(L_G(x))$ and $f_G(w) = F(L_G(x'))$. By definition of $\theta$ this implies that $f_G(v) = \theta(v) = \theta(w) = f_G(w)$ and hence that $v=w$ when $(G,L_G)$ is follower-separated.
\end{proof}

\subsection{Proof of Krieger's theorem, Theorem \ref{Krieger}}

Let $\psi: Y \to Z$ be a conjugacy of sofic shifts. The future cover $(X_{\mathbb K(Y)},L_{\mathbb K(Y)})$ of $Y$ is isomorphic to the cover $(X_{\mathbb K(Y)},\psi \circ L_{\mathbb K(Y)})$ of $Z$ in the sense that
\begin{equation*}
\begin{xymatrix}{
X_{\mathbb K(Y)} \ar@{=}[r] \ar[d]_{L_{\mathbb K(Y)}} & X_{\mathbb K(Y)} \ar[d]^-{\psi \circ L_{\mathbb K(Y)}}  \\
Y \ar[r]^-\psi & Z }
\end{xymatrix}
\end{equation*}
commutes. Note that $\psi \circ L_{\mathbb K(Y)}$ is right-closing because $(\mathbb K(Y),L_{\mathbb K(Y)})$ is right-resolving and $\psi$ is a conjugacy. By Lemma \ref{05-05-25x} there is therefore a right-resolving labeled graph $(G,L_G)$ and a conjugacy $\phi :  X_{\mathbb K(Y)} \to X_G$ such that
\begin{equation*}
\begin{xymatrix}{
X_{\mathbb K(Y)}  \ar[r]^{\phi} \ar[d]_{\psi \circ L_{\mathbb K(Y)}} & X_{G} \ar[d]^-{L_G}  \\
Z \ar@{=}[r] & Z &}
\end{xymatrix}
\end{equation*}
commutes. Hence also
\begin{equation*}
\begin{xymatrix}{
X_{\mathbb K(Y)}   \ar[r]^{\phi} \ar[d]_{L_{\mathbb K(Y)}} & X_{G} \ar[d]^-{L_G}  \\
Y \ar[r]^-\psi & Z &}
\end{xymatrix}
\end{equation*}
commutes. If $L_{\mathbb K(Y)} : \mathbb U(x) \to  \mathbb U(L_{\mathbb K(Y)}(x))$ is surjective, then so is $L_G : \mathbb U(\phi(x)) \to   \mathbb U(L_G(\phi(x))) =\mathbb U(\psi(L_{\mathbb K(Y)}(x)))$, showing that 
\begin{equation}\label{02-09-25ex}
\phi\left(\mathcal R(L_{\mathbb K(Y)} )\right) \subseteq \mathcal R(L_G).
\end{equation}
Let $v\in V_G$ and consider an element $x \in X_G$ with $s_G(x_0) = v$. Then $x = \phi(z)$, where $z = \phi^{-1}(x)$. Since $\phi$ is a conjugacy of shift spaces, it is given by a sliding block code defined on blocks of a given size, say $D \in \mathbb N$, cf. \S 1.5 in \cite{LM}. It follows that $v = s_G(\phi(z')_0)$ for all $z' \in X_{\mathbb K(Y)}$ with the property that $z'_i = z_i$ for $i > -D$. By Corollary \ref{12-09-25} there is a regular ray $z' \in \mathcal R(L_{\mathbb K(Y)})$ such that $z'_i = z_i, \ i > -D$. Then $\phi(z') \in \mathcal R(L_G)$ by \eqref{02-09-25ex} and hence $v = s_G(\phi(z)_0) = s_G(\phi(z')_0)$ is regular by Lemma \ref{12-09-25b}. This shows that $(G,L_G)$ is regular. It follows then from Lemma \ref{02-09-25dx} that there is a labeled-graph homomorphism $\theta_0 : (G,L_G) \to (\mathbb K(Z),L_{\mathbb K(Z)})$ such that $\theta_0 (\mathcal R(L_G)) \subseteq \mathcal R(L_{\mathbb K(Z)})$ and
\begin{equation*}
\begin{xymatrix}{
X_{G}   \ar[r]^{\theta_0} \ar[d]_{L_G} & X_{\mathbb K(Z)} \ar[d]^-{L_{\mathbb K(Z)}}  \\
Z \ar@{=}[r] & Z }
\end{xymatrix}
\end{equation*}
commutes. All together we get a sliding block code $\phi' := \theta_0 \circ \phi :   X_{\mathbb K(Y)} \to X_{\mathbb K(Z)}$ such that 
\begin{equation*}\label{02-09-25gx}
\phi'\left(\mathcal R(L_{\mathbb K(Y)}) \right) \subseteq \mathcal R(L_{\mathbb K(Z)})
\end{equation*}
and
\begin{equation*}
\begin{xymatrix}{
X_{\mathbb K(Y)}   \ar[r]^{\phi'} \ar[d]_{L_{\mathbb K(Y)}} & X_{\mathbb K(Z)} \ar[d]^-{L_{\mathbb K(Z)}}  \\
Y \ar[r]_\psi & Z }
\end{xymatrix}
\end{equation*}
commutes.

Now apply this with $\psi$ replaced by $\psi^{-1}$. This gives us a sliding block code $\phi'' :   X_{\mathbb K(Z)} \to X_{\mathbb K(Y)}$ such that 
\begin{equation*}\label{02-09-25h}
\phi''\left(\mathcal R(L_{\mathbb K(Z)} )\right) \subseteq \mathcal R(L_{\mathbb K(Y)})
\end{equation*}
and
\begin{equation*}
\begin{xymatrix}{
X_{\mathbb K(Z)}   \ar[r]^{\phi''} \ar[d]_{L_{\mathbb K(Z)}} & X_{\mathbb K(Y)} \ar[d]^-{L_{\mathbb K(Y)}}  \\
Z \ar[r]_{\psi^{-1}} & Y }
\end{xymatrix}
\end{equation*}
commutes. Then
\begin{equation*}\label{02-09-25hx}
\phi'' \circ \phi'\left(\mathcal R(L_{\mathbb K(Y)}) \right) \subseteq \mathcal R(L_{\mathbb K(Y)})
\end{equation*}
and
\begin{equation*}
\begin{xymatrix}{
X_{\mathbb K(Y)}   \ar[r]^{\phi''\circ \phi'} \ar[d]_{L_{\mathbb K(Y)}} & X_{\mathbb K(Y)} \ar[d]^-{L_{\mathbb K(Y)}}  \\
Y \ar@{=}[r] & Y }
\end{xymatrix}
\end{equation*}
commutes. Let $x \in \mathcal R(L_{\mathbb K(Y)})$. Then $L_{\mathbb K(Y)}(x) = L_{\mathbb K(Y)}(\phi'' \circ \phi'(x))$ and $\phi'' \circ \phi'(x) \in \mathcal R(L_{\mathbb K(Y)})$. Since $(\mathbb K(Y),L_{\mathbb K(Y)})$ is right-resolving and follower-separated, Lemma \ref{02-09-25bx} implies that $x =\phi'' \circ \phi'(x)$. Since $\mathcal R(L_{\mathbb K(Y)})$ is dense in $X_{\mathbb K(Y)}$ by Corollary \ref{12-09-25}, it follows that $\phi'' \circ \phi' = \id$. By exchanging the roles of $Y$ and $Z$ it follows that also $\phi' \circ \phi'' = \id$. Set $\psi_{\mathbb K} := \phi'$.

Uniqueness: If $\phi_i : \ X_{\mathbb K(Y)} \to  X_{\mathbb K(Z)}$, $i =1,2$, are conjugacies such that $L_{\mathbb K(Z)} \circ \phi_i = \psi \circ L_{\mathbb K(Y)}$, then
$$
\phi_i\left(\mathcal R(L_{\mathbb K(Y)})\right) = \mathcal R(L_{\mathbb K(Z)}).
$$
Let $x \in \mathcal R(L_{\mathbb K(Y)})$. Since $L_{\mathbb K(Z)} \circ \phi_1(x) = \psi \circ L_{\mathbb K(Y)}(x) = L_{\mathbb K(Z)} \circ \phi_2(x)$ and $L_{\mathbb K(Z)}$ is injective on $\mathcal R(L_{\mathbb K(Z)})$ by Lemma \ref{02-09-25bx}, it follows that $\phi_1(x) = \phi_2(x)$. Since $\mathcal R(L_{\mathbb K(Y)})$ is dense in $X_{\mathbb K(Y)}$ by Corollary \ref{12-09-25}, it follows that $\phi_1 = \phi_2$.

Finally, to get the commuting diagram \eqref{24-11-24x}, set $\gamma:= \psi_\mathbb K^{-1} \circ \alpha_{Z} \circ \psi$. By using Lemma \ref{11-09-25g}, and that $\psi_\mathbb K^{-1}$ and $\psi$ both are conjugacies, we find that $\gamma(Y) \subseteq \mathcal R(L_{\mathbb K(Y)})$. Since
$$
L_{\mathbb K(Y)}\circ \gamma = \psi^{-1} \circ L_{\mathbb K(Z)} \circ \alpha_{Z} \circ \psi = \psi^{-1} \circ \psi = \id_Y,
$$
it follows from Lemma \ref{12-09-25j} that $\gamma = \alpha_Y$. \qed

As noted in the introduction an alternative proof of Theorem \ref{Krieger}, without the uniqueness part and the diagram \eqref{24-11-24x}, was given by Nasu in \cite{N}.

\subsection{The terminal component in the future cover of an irreducible sofic shift}

The material in this section is probably folklore to experts, but much of it I haven't found in the literature.

Let $H$ be a finite directed graph without sinks or sources. A vertex $v \in V_H$ is \emph{recurrent} when there is a path $\gamma$ in $H$ such that $s_H(\gamma) = t_H(\gamma) = v$, and it is called \emph{transient} otherwise. Let $V^r_H$ denote the set of recurrent vertices and $V^t_H$ the set of transient vertices. Two vertices $v,w \in V_H$ \emph{communicate}, and we write $v \rightsquigarrow w$, when there is a path $\gamma$ in $H$ with $s_H(\gamma) = v$ and $t_H(\gamma) = w$. When $v,w \in V_H^r, \  v \rightsquigarrow w$ and $w \rightsquigarrow v$, we write $v \sim w$. This is an equivalence relation $\sim$ on $V_H^r$ and an element $C \in V_H^r/\sim$ defines a strongly connected subgraph $H_C$ of $H$ with $V_{H_C} = C$ and $E_{H_C} = \left\{e \in E_H: \ s_H(e), \ t_H(e) \in C\right\}$. The elements of $ V_H^r/\sim$ and the corresponding graphs $H_C, \ C \in V_H^r$, will be referred to as the \emph{components} in $H$ and the irreducible subshifts $X_C := X_{H_C}, \ C \in V_H^r/\sim$, as the components of $X_H$. 
 
 The components $C \in V_H^r/\sim$ are vertices in a directed graph $C_H$ where there is an arrow from $C$ to $C'$ when $v \rightsquigarrow w$ for some elements $v\in C, \ w \in C'$. While there are no sinks or sources in $H$ by assumption, the graph $C_H$ of its communicating components always has both. A \emph{terminal component} (\emph{source component}) in $H$ is a sink (source) $C$ in $C_H$ and the associated subshift $X_C$ is a \emph{terminal component} (\emph{source component}) in $X_H$. A sliding block code $\phi : X_G \to X_H$ between edge shifts respects the components in such a way that for every component $C$ in $G$ there a unique component $\phi(C)$ in $H$ such that $\phi(X_C) \subseteq X_{\phi(C)}$. When $\phi : X_G \to X_H$ is a conjugacy, $\phi(C)$ is a terminal (source) component in $H$ if and only if $C$ is a  terminal (source) component in $G$.

Consider then a labeled graph $(H,L_H)$. A subset $U \subseteq V_H$ of vertices in $H$ is a \emph{hereditary} subset when $e \in E_H, \ s_H(e) \in U \ \Rightarrow \ t_H(e) \in U$. When this holds we let $H_U$ denote the graph with $V_{H_U} := U$ and $E_{H_U} := \left\{ e\in E_H: \ s_H(e) \in U\right\}$. By restriction the labeling $L_H$ defines a labeling of $H_U$ and we denote the resulting labeled graph by $(H_U,L_H)$. We call this a \emph{hereditary labeled subgraph} of $(H,L_H)$. Note that a terminal component $C \subseteq V_H$ of $H$ is a hereditary subset.

\begin{lemma}\label{01-10-25} Let $(H,L_H)$ be a right-resolving and follower-separated labeled graph and $C$ a terminal component in $H$. Then $(H_C,L_H)$ is isomorphic, as a labeled graph, to the minimal right-resolving presentation of $L_H(X_C)$
\end{lemma}
\begin{proof} Note that $(H_C,L_H)$ is  right-resolving and follower-separated. The lemma follows therefore from Corollary 3.3.19 in \cite{LM}. 
\end{proof}

Let $Y$ be a sofic subshift.
When $w \in \mathbb W(Y)$, set
$$
F(w) := \left\{z \in Y[0,\infty): \ wz \in Y[0,\infty) \right\} .
$$

A dual version of the following lemma has appeared as Lemma 4.7 of \cite{J}.

\begin{lemma}\label{27e-09-25} For every $y \in Y$ there is an $N \in \mathbb N$ such that 
$$
F(y_{[-n,-1]})= F(y)
$$
for all $n \geq N$.
\end{lemma}
\begin{proof} Note that $F(y_{[-n-1,-1]}) \subseteq F(y_{[-n,-1]})$ for all $n$ and that
$$
F(y) = \bigcap_n F(y_{[-n,-1]}) .
$$
Since $\left\{F(w) : \ w \in \mathbb W(Y)\right\}$ is a finite set there must be a sequence $n_1 < n_2 < \cdots$ in $\mathbb N$ such that $F(y_{[-n_i,-1]}) =F(y_{[-n_1,-1]})$ for all $i$. Set $N := n_1$. 
\end{proof}

 A word $w \in \mathbb W(Y)$ is \emph{synchronizing} when 
$u,v \in \mathbb W(Y), \ uw,wv \in \mathbb W(Y) \ \Rightarrow \ uwv \in \mathbb W(Y)$.
\begin{lemma}\label{01-10-25a} Let $(H,L_H)$ be a right-resolving, regular and follower-separated labeled graph and $C$ a terminal component in $H$. For each $v \in C$ there is a finite path $\gamma$ in $H$ such that $t_H(\gamma) = v$ and such that every path in $H$ with label $L_H(\gamma)$ terminates at $v$. In particular, $L_H(\gamma)$ is synchronizing for $L_H(X_H)$.
\end{lemma}
\begin{proof} Let $v \in C$. Since $(H,L_H)$ is regular there is a ray $z \in X_H$ such that $t_H(z_{(-\infty,-1]}) = v$ and $f_H(v) = F(L_H(z))$. By Lemma \ref{27e-09-25} there is an $N \in \mathbb N$ such that $F(L_H(z_{[-N,-1]})) = F(L_H(z))$. By Proposition 3.3.16 in \cite{LM} there is a word $u \in \mathbb W(L_H(X_H))$ such that the paths in $H$ with label $L_H(z_{[-N,-1]})u$ all terminate at the same vertex. Since $u$ is a prefix of an element from $F(L_H(z)) = f_H(v)$ there is a path $\gamma_1$ in $H$ such that $s_H(\gamma_1) =v$ and $L_H(\gamma_1) = u$. Note that $t_H(\gamma_1) \in C$ since $C$ is a terminal component. Since $H_C$ is strongly connected there is a path $\gamma_2$ in $H_C$ such that $s_H(\gamma_2) = t_H(\gamma_1)$ and $t_H(\gamma_2) = v$. Set $\gamma :=z_{[-N,-1]}\gamma_1\gamma_2$ and note that every path in $H$ with label $L_H(\gamma)$ must terminate at $v$ since $L_H(z_{[-N,-1]})u$ is a prefix of $L_H(\gamma)$ and $(H,L_H)$ is right-resolving. 
\end{proof}

\begin{prop}\label{29-09-25bx} Let $Y$ be an irreducible sofic shift. There is only one terminal component $C$ in $\mathbb K(Y)$ and $(\mathbb K(Y)_C,L_{\mathbb K(Y)})$ is isomorphic to the minimal right-resolving presentation of $Y$.
\end{prop}
\begin{proof} Let $v,w$ both be vertices of a terminal component in $\mathbb K(Y)$. We must show that they lie in the same component. By Lemma \ref{01-10-25a} there are synchronizing words $s,t$ for $Y$ such that $v = F(s)$ and $w = F(t)$. Let $u \in \mathbb W(Y)$. Since $Y$ is irreducible there is a word $x \in \mathbb W(Y)$ such that $ u =x_{[i,j]}$ for some $i \leq j$ and $sxt \in \mathbb W(Y)$. By definition of $\mathbb K(Y)$ there is therefore a path in $\mathbb K(Y)$ with label $xt$ from $v=F(s)$ to $w = F(t)$. It follows that $v$ and $w$ lie in the same component $C$ of $\mathbb K(Y)$. Furthermore, since the word $u$ was arbitrary it follows that $L_{\mathbb K(Y)}(X_C) = Y$. It follows then from Lemma \ref{01-10-25} that $(\mathbb K(Y)_C,L_{\mathbb K(Y)})$ is isomorphic to the minimal right-resolving presentation of $Y$.
\end{proof}

In the following we let $(\mathbb F(Y),L_{\mathbb F(Y)})$ be the minimal right-resolving presentation of an irreducible sofic shift $Y$, also known as the \emph{Fischer cover} of $Y$. The conclusion of Proposition \ref{29-09-25bx} can then be written as
$$
(\mathbb F(Y),L_{\mathbb F(Y)}) =  (\mathbb K(Y)_C,L_{\mathbb K(Y)}),
$$
when $C$ denotes the terminal component of $\mathbb K(Y)$.

\begin{lemma}\label{30-09-25a} Let $G$ and $H$ be finite directed graphs with unique terminal components, $C$ and $C'$, respectively. Assume that $\chi : X_G \to X_{H}$ is a conjugacy. It follows that $\chi(X_C) = X_{C'}$.
\end{lemma}
\begin{proof} Among the periodic points of $X_G$ the elements in $X_C$ are characterized by the property that for all $x \in X_G$ there is an element $y \in X_G$ and a natural number $k \in \mathbb N$ such that $\lim_{i \to \infty} d(\sigma^{-i}(x),\sigma^{-i}(y)) = 0$ and $\lim_{i \to \infty} d(\sigma^{i}(y),\sigma^{i+k}(p)) = 0$. There is a similar characterization of the periodic elements of $X_H$ that are contained in $X_{C'}$ and we conclude therefore that $\chi(\Per X_C) \subseteq \Per(X_{C'})$. Since the periodic points of $X_C$ are dense in $X_C$ it follows that $\chi(X_C) \subseteq X_{C'}$. Hence, by symmetry, $\chi(X_C) = X_{C'}$.
\end{proof}

\begin{corollary}\label{30-09-25x} Let $\psi : Y \to Z$ be a conjugacy of irreducible sofic subshifts. Let $C$ and $C'$ be the terminal components in $\mathbb K(Y)$ and $\mathbb K(Z)$, respectively, and let $\psi_{\mathbb K} : X_{\mathbb K(Y)} \to X_{\mathbb K(Z)}$ be the conjugacy from Theorem \ref{Krieger}. Then $\psi_{\mathbb K}(X_C) = X_{C'}$.
\end{corollary}

\begin{corollary}\label{30-09-25bx} (Corollary (2.16) in \cite{Kr1}) Let $\psi : Y \to Z$ be a conjugacy of irreducible sofic subshifts. There is unique conjugacy $\psi_{\mathbb F}: X_{\mathbb F(Y)} \to X_{\mathbb F(Z)}$ such that
\begin{equation}\label{30-09-25c}
\xymatrix{
 X_{\mathbb F(Y)} \ar[d]_-{L_{\mathbb F(Y)}} \ar[r]^-{\psi_{\mathbb F}} &  X_{\mathbb F(Z)} \ar[d]^-{L_{\mathbb F(Z)}} \\
Y \ar[r]_-{\psi}  & Z }
\end{equation}
commutes.
\end{corollary}
\begin{proof} The existence of $\psi_{\mathbb F}$ follows from Corollary \ref{30-09-25x}, Proposition  \ref{29-09-25bx} and Theorem \ref{Krieger}. Let $\chi(Y)$ denote the set of elements $x\in X_{\mathbb F(Y)}$ for which $L_{\mathbb F(Y)}^{-1}(L_{\mathbb F(Y)}(x)) = \{x\}$. It follows from Lemma \ref{01-10-25a} that $\mathbb F(Y)$ contains a finite path $\gamma$ such that every path in $\mathbb F(Y)$ with the label $L_{\mathbb F(Y)}(\gamma)$ terminates at the same vertex as $\gamma$. Since $\mathbb F(Y)$ is strongly connected every ray in $X_{\mathbb F(Y)}$ can be approximated by elements that contain $\gamma$ infinitely often to the left, and these elements are in $\chi(Y)$ because $(\mathbb F(Y),L_{\mathbb F(Y)})$ is right-resolving. Hence $\chi(Y)$ is dense in $X_{\mathbb F(Y)}$.

 If $\phi : X_{\mathbb F(Y)} \to X_{\mathbb F(Z)}$ is a conjugacy such that \eqref{30-09-25c} commutes, $\phi(\chi(Y)) = \chi(Z)$. If $\phi_1$ and $\phi_2$ are both conjugacies that make the diagram \eqref{30-09-25c} commute, and $x \in \chi(Y)$, we have that $\phi_i(x) \in \chi(Z)$ and 
$$
L_{\mathbb F(Z)}(\phi_1(x)) = \psi(L_{\mathbb F(Y)}(x)) = L_{\mathbb F(Z)}(\phi_2(x)),
$$
implying that $\phi_1(x) = \phi_2(x)$. Since $\chi(Y)$ is dense in $X_{\mathbb F(Y)}$ it follows that $\phi_1 = \phi_2$.
\end{proof}

As noted in the introduction, an alternative proof of Corollary \ref{30-09-25bx}, without the uniqueness part, was given in \cite{BKM}.

\section{Additional properties of the future cover}

\subsection{A universal property of the future cover $(\mathbb K(Y),L_{\mathbb K(Y)})$.}

 Lemma \ref{02-09-25dx} can now be improved to the following.

\begin{prop}\label{25-09-25} Assume that $(G,L_G)$ is a right-resolving and regular presentation of the sofic shift $Y$. It follows that there is a labeled graph homomorphism 
$\theta : (G,L_G) \to (\mathbb K(Y), L_{\mathbb K(Y)})$ 
such that $
\theta\left(\mathcal R(L_G)\right) \subseteq \mathcal R(L_{\mathbb K(Y)})$,
and $\theta(V_G)$ is a hereditary subset of $V_{\mathbb K(Y)}$. If, furthermore, $(G,L_G)$ is also follower-separated, $\theta$ is a labeled graph isomorphism of $(G,L_G)$ onto the hereditary labeled subgraph $(\mathbb K(Y)_{\theta(V_G)}, L_{\mathbb K(Y)})$.
\end{prop}
\begin{proof} By Lemma \ref{02-09-25dx} and its proof there is a labeled-graph homomorphism $\theta : (G,L_G) \to (\mathbb K(Y), L_{\mathbb K(Y)})$ such that
$\theta(v) =F(L_G(x))$ where $x \in X_G$, $t_G(x_{(-\infty,-1]}) =v$ and $f_G(v) = F(L_G(x))$. Let $F(\theta(v)) \overset{a}{\to} F(z)$ be a labeled arrow in $\mathbb K(Y)$. Then $a$ is the first letter in an element $w \in F(L_G(x)) = f_G(v)$, and there is therefore a right-infinite path $\gamma$ in $G$ such that $s_G(\gamma) = v$ and $L_G(\gamma) = w$. The first arrow in $\gamma$ is an arrow $e \in E_G$ with $s_G(e) = v$ and $L_G(e) = a$. There is therefore an arrow $F(\theta(v)) \overset{a}{\to} \theta(t_G(e))$ in $\mathbb K(Y)$. Since $(\mathbb K(Y),L_{\mathbb K(Y)})$ is right-resolving this implies that $F(z) = \theta(t_G(e))$, proving that $\theta(V_G)$ is a hereditary subset of $V_{\mathbb K(Y)}$. The remaining assertions follow now from  Lemma \ref{02-09-25dx}.

\end{proof}

\begin{corollary}\label{29-09-25} Let $Y$ be a sofic shift. Any labeled graph which presents $Y$ and is right-resolving, regular and follower-separated is isomorphic to a hereditary labeled subgraph of $(\mathbb K(Y),L_{\mathbb K(Y)})$.
\end{corollary}

\begin{corollary}\label{02-10-25c} Let $Y$ be a sofic shift space. For any presentation $(G,L_G)$ of $Y$ which is right-resolving, regular and follower-separated, the number $\# V_G$ of vertices is less than or equal to the number of follower sets in $Y$. The future cover $(\mathbb K(Y),L_{\mathbb K(Y)})$ is the only presentation of $Y$, up to isomorphism of labeled graphs, which is right-resolving, regular and follower-separated, and for which the number of vertices is equal to the number of follower sets in $Y$.
\end{corollary}

\begin{ex}
If we denote the sofic shift presented by the labeled graph in Figure \ref{04-09-25gxx} by $Y$ the future cover $(\mathbb K(Y),L_{\mathbb K(Y)})$ of $Y$ is isomorphic to the following labeled graph.

\begin{figure}[H]
\begin{equation*}
\begin{tikzpicture}[node distance={25mm}, thick, main/.style = {draw, circle}] 
\node[main] (1) {$a$}; 
\node[main] (2) [below of=1] {$b$}; 
\node[main] (3) [right of=2] {$c$};
\draw[->] (1) to [out=10,in=80,looseness=7] node[above ,pos=0.4] {$1$} (1);
\draw[->] (1) to [out=260,in=100,looseness=1] node[left ,pos=0.4] {$2$} (2);
\draw[->] (1) to [out=280,in=70,looseness=1] node[right ,pos=0.4] {$3$} (2);
\draw[->] (1) to [out=320,in=120,looseness=1] node[above ,pos=0.4] {$4$} (3);
\draw[->] (3) to [out=180,in=0,looseness=1] node[below ,pos=0.5] {$3$} (2);
\draw[->] (3) to [out=200,in=260,looseness=5] node[below ,pos=0.4] {$1$} (3);
\draw[->] (3) to [out=280,in=340,looseness=5] node[below ,pos=0.4] {$4$} (3);
\draw[->] (2) to [out=200,in=270,looseness=5] node[below ,pos=0.4] {$1$} (2);
\end{tikzpicture} 
\end{equation*}
\caption{The future cover of the sofic shift presented by the labeled graph in Figure \ref{04-09-25gxx}.}
\label{03-10-25d}
\end{figure}
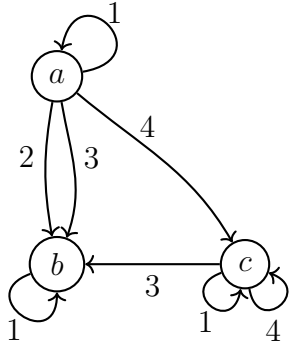
The labeled graph in Figure \ref{04-09-25gxx} is not isomorphic to that in Figure \ref{03-10-25d} although the two graphs have the same number of vertices. Thus regularity is a crucial assumption both in Corollary \ref{02-10-25c}, and indeed in Proposition \ref{25-09-25}.
\end{ex}

\begin{ex}\label{04-10-25} The following labeled graph $(G,L_G)$ is right-resolving, follower-separated and regular (synchronizing, in fact). It is the minimal right-resolving presentation of the  even shift.

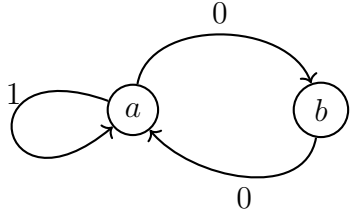
\begin{figure}[H]
\begin{equation*}
\begin{tikzpicture}[node distance={25mm}, thick, main/.style = {draw, circle}] 
\node[main] (1) {$a$}; 
\node[main] (2) [right of=1] {$b$}; 
\draw[->] (1) to [out=80,in=110,looseness=1.0] node[above ,pos=0.5] {$0$} (2); 
\draw[->] (2) to [out=260,in=310,looseness=1.0]  node[below ,pos=0.5] {$0$}(1); 
\draw[->] (1) to [out=160,in=220,looseness=15] node[above ,pos=0.4] {$1$} (1);
\end{tikzpicture} 
\end{equation*}
\caption{A right-resolving, regular and follower-separated presentation of the even shift.}
\label{03-10-25e}
\end{figure}

Note that $(G,L_G)$ is not isomorphic to the future cover of the even shift which is depicted in Figure \ref{11-09-25cx}, and that the number of vertices in $G$ is one less than the number of follower sets in the even shift. This shows that the last condition in Corollary \ref{02-10-25c} is necessary.


\end{ex} 
The future cover has one more property which seems worthwhile to point out; it is predecessor-separated: Consider a labeled graph $(H,L_H)$, and $v \in V_H$ a vertex in $H$. The \emph{predecessor set} of $v$ is the set
$$
p_H(v) := \left\{L_H(z): \ z \in X_H(-\infty,-1], \ t_H(z) = v \right\} .
$$
$(H,L_H)$ is \emph{predecessor-separated} when $p_H(v) = p_H(w) \ \Rightarrow \ v=w$.

\begin{lemma}\label{20-11-25f} Let $(H,L_H)$ be a right-resolving, regular and follower-separated labeled graph. Then $(H,L_H)$ is predecessor-separated.
\end{lemma}
\begin{proof} Assume $v,w \in V_H$ and $p_H(v) = p_H(w)$. Since $v$ is regular there is a $z \in X_H$ such that $t_H(z_{(-\infty,-1]}) =v$ and $f_H(v) = F(L_H(z))$. Then $L_H(z_{(-\infty,-1]}) \in p_H(v) = p_H(w)$, so there is a $u \in X_H(-\infty,-1]$ such that $t_H(u) = w$ and $L_H(u_{(-\infty,-1]}) = L_H(z_{(-\infty,-1]})$. It follows that $f_H(w) \subseteq F(L_H(u)) =F(L_H(z)) = f_H(v)$. By symmetry $f_H(v) \subseteq f_H(w)$ and hence $f_H(v) = f_H(w)$, which implies that $v = w$ since we assume that $(H,L_H)$ is follower-separated.
\end{proof}

\begin{corollary}\label{19-10-25a} The labeled graph $(\mathbb K(Y),L_{\mathbb K(Y)})$ is both follower-separated and predecessor separated.
\end{corollary}
\begin{proof} Combine Lemma \ref{02-10-24fx}, Lemma \ref{04-09-25h} and Lemma \ref{20-11-25f}.
\end{proof}

\subsection{A universal property of the future cover $L_{\mathbb K(Y)} : X_{\mathbb K(Y)} \to Y$.}

By Proposition 4 in \cite{BKM} the minimal right-resolving cover of an irreducible sofic shift has the following universal property:

\begin{prop}\label{20-11-25} Let $Y$ be an irreducible sofic shift. Let $X$ be an irreducible SFT and $\pi : X \to Y$ a right-closing factor map. There is a sliding block code $\rho : X \to X_{\mathbb F(Y)}$ such that
\begin{equation*}
\begin{xymatrix}{
 X \ar[r]^\rho \ar[d]_\pi  & X_{\mathbb F(Y)} \ar[dl]^{L_{\mathbb F(Y)}} \\
  Y &}
\end{xymatrix}
\end{equation*}
commutes.
\end{prop}

This property can be considered as a minimality condition with respect to irreducible and right-closing covers. In general, or more specifically when $Y$ is not of almost finite type, the minimal right-resolving cover is not minimal with respect to arbitrary irreducible covers by Corollary 13 in \cite{BKM}. It is also not minimal with respect to general right-resolving covers; for example the future cover of the even shift which is depicted in Figure \ref{11-09-25cx} does not factor through the minimal right-resolving cover of the even shift. However, its minimality property can be used to characterise it up to isomorphism. For the proof of that we shall use the notion of doubly transitive points. See Definition 9.1.1 in \cite{LM}. We shall also need to know that they are dense in an irreducible SFT. In fact, we have

\begin{lemma}\label{21-06-26} Let $Y$ be an irreducible shift space, as defined in Definition 1.3.6 of \cite{LM}. The doubly transitive points of $Y$ are dense in $Y$.
\end{lemma}
\begin{proof} Let $\{w_i\}_{i=0}^\infty$ be a list of the words in $\mathbb W(Y)$ with every word repeated infinitely often. Let $w$ be an arbitrary word in $\mathbb W(Y)$ of length $2n+1$ for some $n \in \mathbb N$. By using the irreducibility of $Y$ we construct a sequence $\{a_i\}_{i=0}^\infty$ of words in $\mathbb W(Y)$ by induction such that $a_0 = w$ and for each $i$ there are words $b_i,c_i \in \mathbb W(Y)$ with
$$
a_{i+1} = w_ib_ia_ic_iw_i .
$$
Using the compactness of $Y$ this gives us a doubly transitive point $y \in Y$ such that $y_{[-n,n]} = w$.
\end{proof}

\begin{prop}\label{20-11-25a} Let $Y$ be an irreducible sofic shift and $\pi : X \to Y$ a factor map with $X$ an SFT. Then $\pi$ is isomorphic to the minimal right-resolving cover $L_{\mathbb F(Y)} : X_{\mathbb F(Y)} \to Y$ of $Y$ if and only if
\begin{itemize}
\item[a)] $X$ is irreducible,
\item[b)] $\pi$ is right-closing,
\item[c)] $\pi$ is injective on the set of doubly transitive points in $X$, and
\item[d)] $\pi : X \to Y$ is minimal with respect to irreducible right-closing covers of $Y$.
\end{itemize}
\end{prop}
\begin{proof} The following arguments are all implicitly contained in \cite{BKM}. First, it is well-known from \cite{BKM} and \cite{LM} that the minimal right-resolving cover has all four properties, and hence so has any cover isomorphic to it. So assume that $\pi : X \to Y$ is a cover for which they all hold. We must show that $\pi$ is then isomorphic to the minimal right-resolving cover $L_{\mathbb F(Y)} : X_{\mathbb F(Y)} \to Y$ of $Y$. For this note that there are sliding block codes $\rho : X \to X_{\mathbb F(Y)}$ and $\rho' : X_{\mathbb F(Y)} \to X$ such that
\begin{equation*}
\begin{xymatrix}{
X_{\mathbb F(Y)} \ar[r]^-{\rho'} \ar[dr]_-{L_{\mathbb F(Y)}} &X \ar[r]^\rho \ar[d]_\pi  & X_{\mathbb F(Y)} \ar[dl]^{L_{\mathbb F(Y)}} \\
 &  Y &}
\end{xymatrix}
\end{equation*}
and
\begin{equation*}
\begin{xymatrix}{
X \ar[r]^-\rho \ar[dr]_-{\pi} & X_{\mathbb F(Y)} \ar[r]^{\rho'} \ar[d]_{L_{\mathbb F(Y)}}  & X \ar[dl]^{\pi} \\
  & Y &}
\end{xymatrix}
\end{equation*}
both commute. An application of Lemma 9.1.13 in \cite{LM} shows that $\rho$ and $\rho'$ both map doubly transitive points to doubly transitive points, so condition c) implies that $\rho \circ \rho' = \id_{X_{\mathbb F(Y)}}$ and $\rho' \circ \rho = \id_X$ on doubly transitive points. Thanks to Lemma \ref{21-06-26} it follows by continuity that these equations hold all over. 
\end{proof}

We need some preparations in order to discuss a similar characterization of the future cover.

\begin{lemma}\label{17-11-25g} Let $(H,L_H)$ be a regular right-resolving presentation of a sofic shift. For $x \in X_H$ and $N \in \mathbb N$ there is a $y \in \mathcal R(L_H)$ such that $y_i = x_i$ for $i \geq -N$. 
\end{lemma}
\begin{proof} Let $N \in \mathbb N$. Since $s_H(x_{[-N,\infty)})$ is regular there is a $z \in X_H$ such that $t_H(z_{(-\infty,-1]}) = s_H(x_{[-N,\infty)})$ and $f_H(s_H(x_{[-N,\infty)})) = F(L_H(z))$. It follows from Lemma \ref{27e-09-25} that there is an $M \in \mathbb N$ such that $F(L_H(z)) = F(L_H(z_{[-M,-1]}))$. Set $w_1 :=z_{[-M,-1]}$. Since $s_H(w_1)$ is regular we can repeat the argument to get a finite path $w_2$ in $H$ such that $t_H(w_2) = s_H(w_1)$ and $F(L_H(w_2)) = f_H(s_H(w_1))$. Repeat this argument inductively and set
$$
y := \cdots w_3w_2w_1x_{[-N,\infty)} .
$$
For $j \in \left\{-N-\sum_{k =1}^n|w_k|: \ n \in \mathbb N\right\}$ we have by construction that
$$
f_H(t_H(y_{(-\infty,j-1]})) = F(\sigma^j(L_H(y))) .
$$
As in the proof of Lemma \ref{31-08-25x}, by using that $(H,L_H)$ is right-resolving it follows from this that the same holds for all $j \in \mathbb Z$; i.e. $y\in X_H$ is regular.  
\end{proof}

\begin{lemma}\label{17-11-25h} Let $(H,L_H)$ be a right-resolving presentation of a sofic shift. Then $(H,L_H)$ is regular if and only if $\mathcal R(L_H)$ is dense in $X_H$.
\end{lemma}
\begin{proof} One implication follows from Lemma \ref{17-11-25g}. The other from Lemma \ref{12-09-25b}.
\end{proof}

\begin{definition}\label{18-11-25} Let $\pi : X \to Y$ be a sliding block code between subshifts. An element $x \in X$ is \emph{regular} for $\pi$ when $\pi : \mathbb U(x) \to \mathbb U(\pi(x))$ is surjective. The set of 
elements $x \in X$ that are regular for $\pi$ will be denoted by $\mathcal R(\pi)$. We say that $\pi$ is \emph{regular} when $\mathcal R(\pi)$ is dense in $X$. 
\end{definition}

\begin{prop}\label{18-11-25a} Let $\pi : X \to Y$ be a factor code where $X$ is an SFT. Assume that $\pi$ is right-closing and regular. There is a sliding block code $\rho : X \to X_{\mathbb K(Y)}$ such that 
\begin{equation*}
\begin{xymatrix}{
 X \ar[r]^\rho \ar[d]_\pi  & X_{\mathbb K(Y)} \ar[dl]^{L_{\mathbb K(Y)}} \\
  Y &}
\end{xymatrix}
\end{equation*}
commutes.
\end{prop}
\begin{proof} By Lemma \ref{05-05-25x} there is a right-resolving labeled graph $(G,L_G)$ and a conjugacy $\psi : X \to X_G$ such that
\begin{equation*}
\begin{xymatrix}{
 X \ar[r]^\psi \ar[d]_\pi & X_G \ar[dl]^{L_G} \\
  Y &}
\end{xymatrix}
\end{equation*}
commutes. Note that $\psi(\mathcal R(\pi)) = \mathcal R(L_G)$. It follows from Lemma \ref{17-11-25h} that $(G,L_G)$ is regular since $\pi$ is. It follows then from Lemma \ref{02-09-25dx} that there is a sliding block code $\mu : X_G \to X_{\mathbb K(Y)}$ such that
\begin{equation*}
\begin{xymatrix}{
 X_G \ar[r]^\mu \ar[d]_{L_G}  & X_{\mathbb K(Y)} \ar[dl]^{L_{\mathbb K(Y)}} \\
  Y &}
\end{xymatrix}
\end{equation*}
commutes. Set $\rho := \mu \circ \psi$.
\end{proof}


\begin{lemma}\label{18-11-25b} Let $\pi : X \to Y$, $\pi': X' \to Y$ and $\rho : X \to X'$ be sliding block codes such that $\pi' \circ \rho = \pi$. Then $\rho(\mathcal R(\pi)) \subseteq \mathcal R(\pi')$.
\end{lemma}
\begin{proof} Let $x \in \mathcal R(\pi)$ and consider an element $y \in \mathbb U(\pi'(\rho(x))) = \mathbb U(\pi(x))$. There is then an element $z \in \mathbb U(x)$ such that $\pi(z) = y$. Since $\rho(z) \in \mathbb U(\rho(x))$ and $\pi'(\rho(z)) = y$, it follows that $\rho(x) \in \mathcal R(\pi')$.
\end{proof}

\begin{prop}\label{20-11-25ax} Let $\pi : X \to Y$ be a factor map with $X$ an SFT. Then $\pi$ is isomorphic to the future cover $L_{\mathbb K(Y)} : X_{\mathbb K(Y)} \to Y$ of $Y$ if and only if
\begin{itemize}
\item[a)] $\pi$ is regular,
\item[b)] $\pi$ is right-closing,
\item[c)] $\pi$ is injective on the set $\mathcal R(\pi)$ of regular points,
and
\item[d)] $\pi : X \to Y$ is minimal with respect to regular and right-closing factor maps.
\end{itemize}
\end{prop}
\begin{proof} $L_{\mathbb K(Y)}$ is regular by Lemma \ref{17-11-25h} and Lemma \ref{04-09-25h}, and right-closing because the future cover is right-resolving. By Proposition \ref{18-11-25a} it is minimal with respect to regular and right-closing factor maps. That $L_{\mathbb K(Y)}$ is injective on $\mathcal R(L_{\mathbb K(Y)})$ follows from Lemma \ref{02-09-25bx}.

Assume $\pi : X \to Y$ is a factor map with $X$ an SFT such that a), b), c) and d) hold. It follows that there are sliding block codes $\rho : X \to X_{\mathbb K(Y)}$ and $\rho' : X_{\mathbb K(Y)} \to X$ such that
\begin{equation*}
\begin{xymatrix}{
X_{\mathbb K(Y)}\ar[r]^-{\rho'} \ar[dr]_-{L_{\mathbb K(Y)}} &X \ar[r]^\rho \ar[d]_\pi  & X_{\mathbb K(Y)} \ar[dl]^{L_{\mathbb K(Y)}} \\
 &  Y &}
\end{xymatrix}
\end{equation*}
and
\begin{equation*}
\begin{xymatrix}{
X \ar[r]^-\rho \ar[dr]_-{\pi} & X_{\mathbb K(Y)}\ar[r]^{\rho'} \ar[d]_-{L_{\mathbb K(Y)}}  & X \ar[dl]^{\pi} \\
  & Y &}
\end{xymatrix}
\end{equation*}
both commute. It follows from Lemma \ref{18-11-25b} that $\rho$ and $\rho'$ both map regular points to regular points. Since $\pi$ and $L_{\mathbb K(Y)}$ are regular it follows that $\rho \circ \rho' = \id_{X_{\mathbb K(Y)}}$ and $\rho' \circ \rho = \id_X$, first on regular points and then all over by continuity. 
\end{proof}

\subsection{The relation between the future cover and the follower set graph}\label{REG}

In this section we show how the future cover is related to another commonly used presentation of sofic shifts, \emph{the follower set graph}, cf. page 73 in \cite{LM}.

Given a right-resolving labeled graph $(H,L_H)$, the set of its regular vertices is a hereditary subset by Lemma \ref{31-08-25x} and we denote the corresponding hereditary labeled subgraph by $(H_{reg},L_H)$. As shown by the example depicted in Figure \ref{04-09-25gxx}, the sofic shift presented by $(H_{reg},L_H)$ can be smaller than that presented by $(H,L_H)$.

\begin{prop}\label{16-10-25e}
Let $Y$ be a sofic shift and $(G,\mathcal L)$ the follower set graph of $Y$. The future cover $(\mathbb K(Y),L_{\mathbb K(Y)})$ of $Y$ is isomorphic to $(G_{reg},\mathcal L)$ as a labeled graph.
\end{prop}
\begin{proof} The proof uses the observation that two words $w_1,w_2 \in \mathbb W(Y)$ have the same follower sets as defined in Definition 3.2.4 of \cite{LM} if and only if $F(w_1) = F(w_2)$. Let $y \in Y$. By Lemma \ref{27e-09-25} there is an $N \in \mathbb N$ such that  $F(y) = F(y_{[-n,-1]})$ for all $n \geq N$. If $z \in Y$ and $F(y) = F(z)$ there is an $M \in \mathbb N$ such that $F(z_{[-n,-1]}) = F(y_{[-n,-1]})$ for all $n \geq M$. It follows that we can define a map $\chi: V_{\mathbb K(Y)} \to V_G$ such that $\chi(F(y)) = F(y_{[-n,-1]})$ for all large $n \in \mathbb N$. This map, $\chi$, is clearly injective.  For $y,z \in Y$ there is a labeled arrow $F(y) \overset{a}{\to} F(z)$ in $(\mathbb K(Y),L_{\mathbb K(Y)})$ if and only if $F(z) = \frac{F(y)}{a}$ if and only if $F(z_{[-n,-1]}) = \frac{F(y_{[-n,-1]})}{a}$ for all large $n$ if and only if there is a labeled arrow $\chi(F(y)) \overset{a}{\to} \chi(F(z))$ in $G$. This shows that $\chi$ is an embedding of labeled graphs. To show that $\chi(V_{\mathbb K(Y)})$ is a hereditary subset, let $w \in \mathbb W(Y)$ and assume that $\chi(F(y)) \overset{a}{\to} F(w)$ in $G$. Then $F(w) = \frac{F(y)}{a}$ since $(G,\mathcal L)$ is right-resolving, and hence $F(w) = F(z)$, where $z \in Y$ is any element with $z_{(-\infty,-1]} = y_{(-\infty,-1]}a$. It follows that $F(w) = \chi(F(z))$. 

Let $y \in Y$. Then $z :=\chi(\alpha_Y(y)) \in X_G$ and $t_G(z_{(-\infty,-1]}) = \chi(F(y))$. Note that $L_G(z) = y$. Hence $F(L_G(z)) = F(y)$. So if $w \in F(L_G(z))$ there is an $x \in X_{\mathbb K(Y)}[0,\infty)$ such that $s_{\mathbb K(Y)}(x) = F(y)$ and $L_{\mathbb K(Y)}(x) = w$. Then $\chi(x) \in X_G[0,\infty)$, $s_{G}(\chi(x)) = \chi(F(y))$ and $L_G(\chi(x)) =  w$. This shows that $F(L_G(z)) \subseteq f_G(\chi(F(y)))$, implying that $\chi(F(y))$ is a regular vertex in $G$. Thus all vertices of $\chi(V_{\mathbb K(Y)})$ are regular. Conversely, if $w \in \mathbb W(Y)$ and $F(w)$ is a regular vertex in $G$, it follows from the defining relation \eqref{11-09-25} that $F(w) = F(z)$ for some $z \in Y$, proving that $F(w) = \chi(F(z)) \in \chi(V_{\mathbb K(Y)})$.
\end{proof}


\begin{ex} Consider the following labeled graph.

\begin{figure}[H]
\begin{equation*}
\begin{tikzpicture}[node distance={25mm}, thick, main/.style = {draw, circle}] 
\node[main] (1) {$a$}; 
\node[main] (2) [right of=1] {$b$}; 
\node[main] (3) [right of=2] {$c$};
\node[main] (4) [below of=3] {$d$};
\node[main] (5) [below of=2] {$e$};
\draw[->] (1) to [out=120,in=70,looseness=8] node[above ,pos=0.4] {$0$} (1);
\draw[->] (1) to [out=0,in=180,looseness=1] node[above ,pos=0.4] {$2$} (2);
\draw[->] (2) to [out=120,in=70,looseness=8] node[above,pos=0.4] {$0$} (2);
\draw[->] (2) to [out=0,in=180,looseness=1] node[above ,pos=0.4] {$3$} (3);
\draw[->] (3) to [out=270,in=90,looseness=1] node[right ,pos=0.4] {$1$} (4);
\draw[->] (4) to [out=290,in=250,looseness=8] node[below ,pos=0.4] {$0$} (4);
\draw[->] (4) to [out=180,in=0,looseness=1] node[below ,pos=0.4] {$4$} (5);
\draw[->] (5) to [out=120,in=300,looseness=1] node[left ,pos=0.4] {$1$} (1);
\end{tikzpicture} 
\end{equation*}
\caption{The minimal right-resolving presentation of a mixing sofic shift whose follower set graph is larger than its future cover.}
\label{16-10-25}
\end{figure}
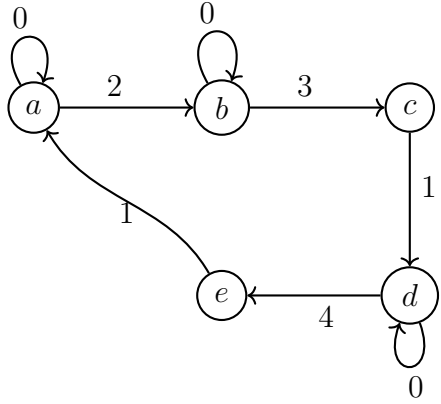

The follower sets of the vertices $a,b,c,d,e$ are also follower sets of (synchronizing) words in $\mathbb W(Y)$. Besides these the follower sets of words from $\mathbb W(Y)$ include only the follower sets of the words $10$ and $0$, resulting in the following labeled graph depicting the follower set graph of $Y$.

\begin{figure}[H]
\begin{equation*}
\begin{tikzpicture}[node distance={25mm}, thick, main/.style = {draw, circle}] 
\node[main] (1) {$10$};
\node[main] (2) [right of=1] {$a$}; 
\node[main] (3) [right of=2] {$b$}; 
\node[main] (4) [right of=3] {$c$};
\node[main] (5) [below of=4] {$d$};
\node[main] (6) [below of=3] {$e$};
\node[main] (7) [right of=5] {$0$};
\draw[->] (2) to [out=170,in=250,looseness=8] node[below ,pos=0.4] {$0$} (2);
\draw[->] (2) to [out=0,in=180,looseness=1] node[above ,pos=0.4] {$2$} (3);
\draw[->] (3) to [out=120,in=70,looseness=8] node[above,pos=0.4] {$0$} (3);
\draw[->] (3) to [out=0,in=180,looseness=1] node[above ,pos=0.4] {$3$} (4);
\draw[->] (4) to [out=270,in=90,looseness=1] node[right ,pos=0.4] {$1$} (5);
\draw[->] (5) to [out=70,in=350,looseness=8] node[right ,pos=0.4] {$0$} (5);
\draw[->] (5) to [out=180,in=0,looseness=1] node[above ,pos=0.4] {$4$} (6);
\draw[->] (6) to [out=120,in=300,looseness=1] node[left ,pos=0.4] {$1$} (2);
\draw[->] (1) to [out=40,in=150,looseness=1] node[above ,pos=0.4] {$2$} (3);
\draw[->] (1) to [out=270,in=180,looseness=1] node[below ,pos=0.4] {$4$} (6);
\draw[->] (7) to [out=100,in=45,looseness=2] node[above,pos=0.4] {$2$} (3);
\draw[->] (7) to [out=110,in=340,looseness=1] node[left,pos=0.4] {$3$} (4);
\draw[->] (7) to [out=210,in=340,looseness=1] node[below,pos=0.4] {$4$} (6);
\draw[->] (7) to [out=30,in=330,looseness=5] node[right,pos=0.4] {$0$} (7);
\draw[->] (1) to [out=180,in=110,looseness=4] node[left,pos=0.4] {$0$} (1);
\end{tikzpicture} 
\end{equation*}
\caption{The follower set graph of the sofic shift presented by Figure \ref{16-10-25}.}
\label{16-10-25c}
\end{figure}
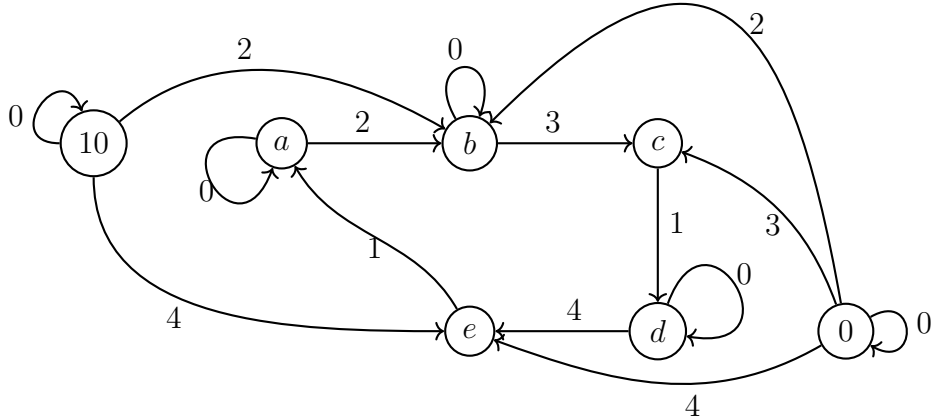
The vertex $10$ is not regular while all the other are, and hence the future cover of $Y$ is the labeled graph obtained from the graph in Figure \ref{16-10-25c} by deleting the vertex $10$ and the arrows it emits.

\end{ex}

\section{An addendum to Krieger's theorem}

\subsection{Subset constructions }

In this section by a \emph{dynamical system} we mean a pair $(X,\phi)$ where $X$ is a compact metrizable space and $\phi : X \to X$ is a homeomorphism of $X$ onto itself. 

Let $(X,\phi)$ be a dynamical system and $d$ a metric for the topology of $X$. The set of non-empty compact subsets $K$ of $X$ can be equipped with a metric, the \emph{Hausdorff metric} $d_{Haus}$, defined by
$$
d_{Haus}(K,K') := \max \left\{ \sup_{x \in K}\left(\inf_{y \in K'}d(x,y)\right), \  \sup_{y \in K'}\left(\inf_{x \in K}d(x,y)\right)\right\}.
$$
This turns the set of non-empty compact subsets of $X$ into a compact metric space 
$$(X^{\#},d_{Haus}).
$$
See for example Section 2.6 in \cite{B}.

When $(X,d)$ and $(X',d')$ are two compact metric spaces and $\psi : X \to X'$ is a continuous map we get a continuous map $\psi^{\#} : X^{\#} \to X'^{\#}$ defined such that
$$
\psi^{\#}(K) := \psi(K) 
$$ 
for $K \in X^\#$. In particular, we get from $(X,\phi)$ a new dynamical system $(X^{\#},\phi^\#)$. More generally, when $(X,\phi)$ and $(Y,\phi')$ are dynamical systems and $\pi : X \to Y$ is a map of dynamical systems in the sense that $\pi \circ \phi = \phi' \circ \pi$, the set
\begin{equation*}\label{23-06-25b}
X^\pi := \left\{ K \in X^{\#}: \ \pi(x) = \pi(y)\ \forall x,y \in K \right\} 
\end{equation*}
is a closed $\phi^\#$-invariant subset of $X^\#$ and $(X^\pi,\phi^\#)$ is also a dynamical system. When $\pi$ is surjective and hence a factor map, then so is $\pi^\# : (X^\pi,\phi^\#) \to (Y,\phi')$.


The following is now a straightforward observation.
\begin{lemma}\label{25-06-25ax}
Consider two conjugacies $\psi' : (X,\phi) \to (X',\phi')$ and $\psi : (Y,\mu) \to (Y',\mu')$ of dynamical systems, and factor maps $\pi : (X,\phi) \to (Y,\mu)$ and $\pi' : (X',\phi') \to (Y',\mu')$ such that
\begin{equation*}\label{23-06-25x}
\xymatrix{
 X \ar[r]^-{\psi'} \ar[d]_\pi &  X' \ar[d]^{\pi'} \\
Y  \ar[r]^-\psi & Y'}
\end{equation*}
commutes. Then
$\psi'^\# : X^\pi \to {X'}^{\pi'}$
is a conjugacy, and 
\begin{equation*}\label{25-06-25bx}
\xymatrix{
 X^\pi \ar[r]^-{\psi'^\#} \ar[d]_{\pi^\#} &  X'^{\pi'} \ar[d]^{\pi'^\#} \\
Y  \ar[r]^-\psi & Y'}
\end{equation*}
commutes.
\end{lemma}

\subsection{Subset constructions for sofic shifts}\label{subset}

Let $(H,L_H)$ be a labeled graph and $Y:= L_H(X_H)$ the presented sofic shift. We define a labeled graph $(H'', L_{H''})$ as follows. The set $V_{H''}$ of vertices consists of the non-empty subsets $F$ of vertices from $V_H$. For $F,F' \in  V_{H''}$ there is an edge $e' \in E_{H''}$ with $s_{H''} (e') = F$ and $t_{H''} (e') = F'$ when there is a symbol $a \in A$, the alphabet of $Y$, such that
\begin{equation*}\label{23-06-25ay}
F \subseteq \left\{ s_H(e): \ e \in E_H, \ L_H(e) = a \right\} 
\end{equation*}
and
$$
F' = \left\{ t_H(e) : \ e \in E_H, \ s_H(e) \in F, \ L_H(e) = a \right\} .
$$ 
We set then $L_{H''} (e') = a$. While the graph $H''$ in principle contains all subsets of $V_H$ as vertices, it is clear from the definition of arrows in $E_{H''}$ that typically only a much smaller collection of such subsets will occur in elements of $X_{H''}$. In particular, when we consider $H''$ in the following it will be assumed that $H''$ has been trimmed by repeatedly removing sources and sinks to obtain a graph that does not contain neither.

\begin{lemma}\label{23-06-25dyx} Assume that $(H,L_H)$ is right-resolving. There is a conjugacy $\kappa : (X_H^{L_H}, \sigma^\#) \to (X_{H''} , \sigma)$ such that
\begin{equation*}\label{23-06-25cy}
\xymatrix{
 X^{L_H}_H \ar[rr]^-\kappa \ar[dr]_-{L_H^\#} & &  X_{H''}  \ar[dl]^-{L_{H''} } \\
 & Y & }
\end{equation*}
commutes.
\end{lemma}
\begin{proof} Let $ K \in X^{L_H}_H$. Note that $K$ is a finite set since $(H,L_H)$ is right-resolving. Define 
$$
F_i := \left\{ s_H(k_i) : \ k \in K \right\}
$$
for $i \in \mathbb Z$. For each $i \in \mathbb Z$ there is a unique edge $e_i(K)$ in $E_{H''} $ from $F_i$ to $F_{i+1}$ with label $L_{H''} (e_i(K)) = L_H(k_i)$ for all $k \in K$, and we set
$$
\kappa(K) := \left(e_i(K)\right)_{i \in \mathbb Z} \in X_{H''}  .
$$ 
Since $e_i(\sigma(K)) = e_{i+1}(K)$ we find that 
$$
\kappa \circ \sigma^\#(K) = \kappa(\sigma(K)) = \left(e_i(\sigma(K))\right)_{i \in \mathbb Z} = \left(e_{i+1}(K)\right)_{i \in \mathbb Z} = \left(\sigma(\kappa(K))_i\right)_{i \in \mathbb Z},
$$
proving that $\kappa\circ \sigma^\# = \sigma \circ \kappa$. Note also that $L_{H''}  (\kappa(K)) = L_H^\#(K)$.

To construct an inverse map $\kappa^{-1} : X_{H''}  \to X_H^{L_H}$ to $\kappa$, let $x = (x_i)_{i \in \mathbb Z} \in X_{H''} $. Since $(H,L_H)$ is right-resolving, $\# s_{H''} (x_i) \geq  \# s_{H''} (x_{i+1})$, and we choose $N \in \mathbb N$ such that $\# s_{H''} (x_i) = \# s_{H''} (x_N)$ for all $i \leq N$. For each $v \in  s_{H''} (x_N)$ there is a unique ray $\gamma^v \in X_H$ such that $s_H(\gamma^v_N) = v$, $s_H(\gamma^v_i) \in s_{H''}(x_i)$ for all $i \in \mathbb Z$ and $L_H(\gamma^v) = L_{H''} (x)$. Setting
$$
\kappa^{-1}(x) := \left\{\gamma^v : \ v \in  s_{H''} (x_N) \right\} 
$$
we obtain a well-defined map $\kappa^{-1}: X_{H''}  \to X^{L_H}_H$. It is straightforward to check that $\kappa^{-1}$ is the inverse of $\kappa$.

Since $X_H^{L_H}$ and $X_{H''} $ are compact Hausdorff spaces it suffices now to show that $\kappa$ is continuous. For this we use the metric $d$ on $X_H$ defined such that $d(x,y) = \frac{1}{n+1}$, where $n = \infty$ when $x = y$ and otherwise $n = \min \left\{ j\in \mathbb N \cup \{0\} : \ x_{[-j,j]} \neq y_{[-j,j]}\right\}$. Consider then two elements $K,K' \in X_H^{L_H}$. If $d_{Haus}(K,K') < \frac{1}{N}$, it follows that 
$$
\left\{k_i : \ k \in K\right\} =  \left\{k_i : \ k \in K'\right\}  
$$
for $-N \leq i \leq N$, and hence also that $\kappa(K)_i = e_i(K) = e_i(K') = \kappa(K')_i$ for $-N \leq i \leq N$. This shows that $\kappa$ is continuous.
\end{proof}

\begin{corollary}\label{17-11-25} Let $Y$ and $Z$ be sofic shifts and $(G,L_G)$, $(H,L_H)$ right-resolving presentations of $Y$ and $Z$, respectively. Assume that there are conjugacies $\psi : Y \to Z$ and $\phi : X_G \to X_H$ such that 
\begin{equation*}\label{17-11-25b}
\xymatrix{
 X_G \ar[r]^-{\phi} \ar[d]_{L_G}&  X_H \ar[d]^{L_H} \\
 Y  \ar[r]^-\psi & Z}
\end{equation*}
commutes. There is a conjugacy $\phi'' : X_{G''} \to X_{H''}$ such that
\begin{equation*}\label{17-11-25c}
\xymatrix{
 X_{G''} \ar[r]^-{\phi''} \ar[d]_{L_{G''}}&  X_{H''} \ar[d]^{L_{H''}} \\
Y   \ar[r]^-\psi & Z}
\end{equation*}
commutes.
\end{corollary}
\begin{proof} Combine Lemma \ref{25-06-25ax} and Lemma \ref{23-06-25dyx}.
\end{proof}

Thus $({G''},L_{G''})$ is a weakly canonical cover when $(G,L_G)$ is. However, the conjugacy $\phi''$ may not be unique and there is no analogue to the right-inverse for the labeling map as in Krieger's theorem.  We continue therefore to identify a labeled subgraph of $(G'',L_{G''})$ which does give rise to a strongly canonical cover.

\bigskip

Since $(H,L_H)$ is right-resolving, $L_H^{-1}(y)$ is a finite set and hence an element of $X_H^{L_H}$ for each $y \in Y$. We define $\beta_H : Y \to X_{H''} $ such that
$$
\beta_H(y) := \kappa(L_H^{-1}(y)),
$$
where $\kappa : X_H^{L_H} \to X_{H''}$ is the conjugacy from Lemma \ref{23-06-25dyx}. Note that
\begin{equation}\label{22-11-25}
L_{H''} \circ \beta_H = \id_Y .
\end{equation}
Like the map $\alpha_Y$ considered in Section \ref{rays} the map $\beta_H$ commutes with the shift but is not continuous.

To proceed we need some terminology. Let $X$ be a subshift. We say that two elements $x,y \in X$ are \emph{backward asymptotic} and that \emph{$x$ is backward asymptotic to $y$} when there is an $N \in \mathbb Z$ such that $x_i=y_i$ for all $i \leq N$. When $A \subseteq X$ is a subset and $x \in X$ we say that $x$ is \emph{backward asymptotic to $A$} when $x$ is backward asymptotic to an element of $A$. When $A$ is a subshift of $X$ we say that $x \in X$ is \emph{backward transitive} to $A$ when $x$ is backward asymptotic to $A$ and every word $w \in \mathbb W(A)$ appears infinitely often to the left in $x$, in the sense that for all $N\in \mathbb Z$ there are $i \leq j \leq N$ such that $w =x_{[i,j]}$. 

Concerning the components of $H''$ we will say that a component $C'$ in $H''$ \emph{connects} to another component $C$ when there is a path $\mu$ in $H''$ such that $s_{H''}(\mu) \in C'$ and $t_{H''}(\mu) \in C$. Note that for any component $C$ in $H''$ the cardinality $\# F$ of the elements $F$ of $C \subseteq 2^{V_H}$ is the same; this number is the \emph{multiplicity} of $C$ and we denote it by $M(C)$.
 
\begin{lemma}\label{28-12-25d} Let $C$ be a component in $H''$ such that $\beta_H(y)$ is backward asymptotic to $X_C$ for some $y \in Y$. There is a source component $C'$ in $H''$ with the property that every element of $X_C$ can be approximated by elements of the form $\beta_H(z)$ for some $z \in Y$ with $\beta_H(z)$ backward asymptotic to $X_{C'}$. 
\end{lemma}
\begin{proof} Let $y \in Y$ and let $C$ be the component in $H''$ for which there is a $N \in \mathbb Z$ such that $\beta_H(y)_j \in C$ for all $j \leq N$. We claim that there is an $R < N$ such that if $z \in X_{H''}$ and $L_{H''}(z)_{[R,\infty)} = y_{[R,\infty)}$, then $s_{H''}(z_N) \subseteq s_{H''}(\beta_H(y)_N)$. To show this by contradiction, assume that no such $R$ exists. There is then a sequence $\{z^i\}$ in $X_{H''}$ and a sequence $\{r_i\}$ in $\mathbb Z$ such that $r_{i+1} < r_i < N$, $L_{H''}(z^i)_{[r_i,\infty)} = y_{[r_i,\infty)}$ and $s_{H''}(z^i_N) \nsubseteq s_{H''}(\beta_H(y)_N)$ for all $i$. By passing to a subsequence we may assume that $s_{H''}(z^i_N)$ is the same subset $M$ of $V_H$ for all $i$. By compactness of $X_{H''}$ we get in this way an element $z \in L_{H''}^{-1}(y)$ such that $s_{H''}(z_N) = M \nsubseteq  s_{H''}(\beta_H(y)_N)$, contradicting the definition of $\beta_H(y)$. This proves the claim and we fix now such an $R < N$.



Let $C'$ be a component which connects to $C$ and has the largest multiplicity among the components that connect to $C$. Note that we can arrange that $C'$ is a source component in $H''$. Let $\gamma$ be a finite path in $H''_C$. We can then choose an element $z'  \in X_{H''}$ with the following properties:
\begin{itemize}
\item[i)] $z'$ is backward asymptotic to $X_{C'}$,
\item[ii)] there is a $R' < R$ such that $z'_{[R',R]}$ is a path in $H''_C$ with $\gamma \subseteq z'_{[R',R]}$, and
\item[iii)] $z'_{[R,\infty)} = \beta_H(y)_{[R,\infty)}$.
\end{itemize}
We claim that 
\begin{equation}\label{foerstejuledag1}
z' = \beta_H(L_{H''}(z')).
\end{equation}
To see this, we use that  
\begin{equation}\label{foerstejuledag2} 
\beta_H(L_{H''}(z'))_N =  \beta_H(y)_N.
 \end{equation}
Indeed, $L_{H''}(\beta_H(L_{H''}(z')))_{[R,\infty)} = L_{H''}(z')_{[R,\infty)}= y_{[R,\infty)}$ by iii), which implies that 
 $$
 s_{H''}(\beta_H(L_{H''}(z'))_N) \subseteq  s_{H''}(\beta_H(y)_N)
 $$ 
 by the choice of $R$, while the opposite inclusion follows from 
 $$
 s_{H''}(\beta_H(y)_N) = s_{H''}(z'_N) \subseteq s_{H''}(\beta_H(L_{H''}(z'))_N).
 $$

 Let $C''$ be the component in $H''$ with the property that $\beta_H(L_{H''}(z'))$ is backward asymptotic to $X_{C''}$. Since 
 \begin{equation}\label{29-12-25}
 s_{H''}(z'_j) \subseteq s_{H''}(\beta_H(L_{H''}(z'))_j)
 \end{equation} 
 for all $j$ and $z'$ is backward asymptotic to $X_{C'}$ while $\beta_H(L_{H''}(z'))$ is backward asymptotic to $X_{C''}$, we have that $M(C')  \leq M(C'')$. But it follows from \eqref{foerstejuledag2} that $C''$ connects to $C$ and hence $M(C'') \leq M(C')$ by definition of $C'$, and we find therefore that $M(C'') = M(C')$. It follows from this identity and \eqref{29-12-25} that $s_{H''}(z'_j) = s_{H''}(\beta_H(L_{H''}(z'))_j)$ for all $j$ 'close to $-\infty$'. Since $H''$ is right-resolving, and $z'$ and $\beta_H(L_{H''}(z'))$ carry the same labels, it follows that \eqref{foerstejuledag1} holds. 
 
 For any $k \in \mathbb Z$ the element $\sigma^k(z') = \beta_H(L_{H''}(\sigma^k(z')))$ is backward asymptotic to $X_{C'}$ because $z'$ is. Since $\gamma$ is an arbitrary path in $H''_C$ it follows therefore that every element of $X_C$ can be approximated by elements of the form $\beta_H(z)$ for some $z \in Y$ with $\beta_H(z)$ backward asymptotic to $X_{C'}$. 
 \end{proof}

\begin{definition}\label{30-10-25} We denote by $H'$ the hereditary subgraph of $H''$ generated by the components $C$ of $H''$ for which there is a point $y \in Y$ such that $\beta_H(y)$ is backward asymptotic to $X_C$. 
\end{definition}

That is, $H'$ consists of the vertices in $H''$ that can be reached in $H''$ from one of the mentioned components $C$, together with all arrows emitted from these vertices. To simplify notation, set $L_{H'} := L_{H''}|_{H'}$. Since $\beta_H(Y) \subseteq X_{H'}$ by definition of $H'$, it follows from \eqref{22-11-25} that $L_{X_{H'}}(X_{H'}) = Y$, implying that $(H',L_{H'})$ is a presentation of $Y$.

\begin{thm}\label{31-10-25hx} Let $(G,L_G)$ and $(H,L_H)$ be right-resolving presentations of the sofic shifts $Y$ and $Z$, respectively. Assume that there are conjugacies $\psi : Y \to Z$ and $\phi : X_G \to X_H$ such that 
\begin{equation*}\label{30-10-25bx}
\xymatrix{
 X_G \ar[r]^-{\phi} \ar[d]_{L_G}&  X_H \ar[d]^{L_H} \\
 Y  \ar[r]^-\psi & Z}
\end{equation*}
commutes. There is a unique conjugacy $\phi' : X_{G'} \to X_{H'}$ such that
\begin{equation*}\label{30-10-25cx}
\xymatrix{
 X_{G'} \ar[r]^-{\phi'} \ar[d]_{L_{G'}}&  X_{H'} \ar[d]^{L_{H'}} \\
Y   \ar[r]^-\psi & Z}
\end{equation*}
commutes.
\end{thm}
\begin{proof} It follows from Corollary \ref{17-11-25} that there is a conjugacy $\phi'' : X_{G''} \to X_{H''}$ such that
\begin{equation*}\label{22-11-25a}
\xymatrix{
 X_{G''} \ar[r]^-{\phi''} \ar[d]_{L_{G''}}&  X_{H''} \ar[d]^{L_{H''}} \\
 Y  \ar[r]^-\psi & Z}
\end{equation*}
commutes. When we use the same symbol $\kappa$ for the two conjugacies $\kappa : X_G^{L_G} \to X_{G''}$ and $\kappa : X_H^{L_H} \to X_{H''}$ from Lemma \ref{23-06-25dyx}, we note that $\phi'' = \kappa \circ \phi^\# \circ \kappa^{-1}$. Since $\phi^\#(L_G^{-1}(y)) = L_H^{-1}(\psi(y))$ it follows that 
\begin{equation}\label{30-10-26}
\# L_G^{-1}(y) = \# L_H^{-1}(\psi(y))
\end{equation}
and
 \begin{equation}\label{28-12-25c}
\phi''(\beta_G(y)) = \beta_H(\psi(y))
\end{equation}
 for all $y \in Y$. Let $\mathcal C_G$ denote the set of components $C$ of $G''$ with the property that there is a point $y \in Y$ such that $\beta_G(y)$ is backward asymptotic to $X_C$. We denote the analogous set of components in $H''$ by $\mathcal C_H$. It follows from \eqref{28-12-25c} that when $C \in \mathcal C_G$ we have that $\phi''(X_C) = X_{C'}$ for some $C' \in \mathcal C_H$. If $x \in X_{G''}$ is backward asymptotic to $X_C$ we have that $\phi''(x)$ is backward asymptotic to $X_{C'}$. Since the elements of $X_{G''}$ that are backward asymptotic to $X_C$ for some $C \in \mathcal C_G$ constitute a dense set in $X_{G'}$, it follows that $\phi''(X_{G'}) \subseteq X_{H'}$. By symmetry we have equality; $\phi''(X_{G'}) =X_{H'}$, and we set $\phi':= \phi''|_{X_{G'}}$. This proves the existence part. 
 To prove uniqueness, consider a conjugacy $\lambda : X_{G'} \to X_{H'}$ such that
\begin{equation*}\label{30-10-25cy}
\xymatrix{
 X_{G'} \ar[r]^-{\lambda} \ar[d]_{L_{G'}}&  X_{H'} \ar[d]^{L_{H'}} \\
Y   \ar[r]^-\psi & Z}
\end{equation*}
commutes. Set 
$$
\lambda' := \kappa^{-1}\circ \lambda \circ \kappa|_{\kappa^{-1}(X_{G'})}.
$$ 
By Lemma \ref{23-06-25dyx} also
\begin{equation}\label{28-12-25e}
\xymatrix{
 X_{G}^{L_G}\cap \kappa^{-1}(X_{G'}) \ar[r]^-{\lambda'} \ar[d]_{L_{G}^{\#}}&  X_{H}^{L_H}\cap \kappa^{-1}(X_{H'}) \ar[d]^{L_{H}^{\#}} \\
Y   \ar[r]^-\psi & Z}
\end{equation}
commutes. Note that $L_G^{-1}(y)  \in  X_{G}^{L_G}\cap \kappa^{-1}(X_{G'})$ for all $y \in Y$ since $\beta_G(Y) \subseteq X_{G'}$. It follows from \eqref{28-12-25e} 
\begin{equation}\label{29-10-26}
\lambda'(L_G^{-1}(y)) \subseteq L_H^{-1}(\psi(y)).
\end{equation}
To show that we have equality in \eqref{29-10-26} for sufficiently many $y$, consider a component $C$ in $G'$ and an element $y \in Y$ such that $\beta_G(y)$ is backward asymptotic to $X_C$. Then $\lambda(\beta_G(y))$ is backward asymptotic to $\lambda(C)$ and hence
\begin{equation}\label{31-10-26}
M(\lambda(C)) = \lim_{k \to -\infty} \# s_{H'}(\lambda(\beta_G(y))_k) = \# \kappa^{-1}(\lambda(\beta_G(y))) = \# \lambda'(L_G^{-1}(y)).
\end{equation}
By \eqref{29-10-26} and \eqref{30-10-26}, $\# \lambda'(L_G^{-1}(y))\leq \# L_H^{-1}(\psi(y)) = \# L_G^{-1}(y)$. Since
$$
\# L_G^{-1}(y) = \lim_{k \to -\infty} \# s_{G'}(\beta_G(y)_k) = M(C),
$$
we conclude that
\begin{equation}\label{31-10-26a}
M(\lambda(C)) = \# \lambda'(L_G^{-1}(y)) \leq \# L_H^{-1}(\psi(y)) = M(C).
\end{equation}
By definition of $G'$ every source component of $G'$ is in $\mathcal C_G$ and hence \eqref{31-10-26a} implies that 
\begin{equation}\label{31-10-26b}
M(\lambda(C)) \leq M(C)
\end{equation} 
when $C$ is a source component in $G'$. But then $\lambda(C)$ is a source component in $H'$ so by replacing $G$ by $H$ and $\lambda$ by $\lambda^{-1}$ we find that
$$
M(C) = M(\lambda^{-1}(\lambda(C))) \leq M(\lambda(C)).
$$
Combined with \eqref{31-10-26b} this shows that $M(\lambda(C)) = M(C)$. It follows then from \eqref{31-10-26a} and \eqref{29-10-26} that $\lambda'(L_G^{-1}(y)) = L_H^{-1}(\psi(y))$, or equivalently that
\begin{equation}\label{29-10-26ax}
\lambda(\beta_G(y)) = \beta_H(\psi(y))
\end{equation}
when $\beta_G(y)$ is backward asymptotic to a source component in $G'$. Since $\beta_H(\psi(y)) = \phi'(\beta_G(y))$ by \eqref{28-12-25c} we have therefore shown that $\lambda(\beta_G(y)) = \phi'(\beta_G(y))$ when $\beta_G(y)$ is backward asymptotic to a source component in $X_{G'}$. Note that the component $C'$ from Lemma \ref{28-12-25d} applied to a component $C \in \mathcal C_G$ is a source component in $G'$. It follows therefore now from Lemma \ref{28-12-25d} and the continuity of $\phi'$ and $\lambda$ that $\lambda|_{X_C} = \phi'|_{X_C}$ when $C \in \mathcal C_G$. Consider then an $x \in X_{G'}$ backward asymptotic to $X_C$ for some component $C \in \mathcal C_G$. Since $\lambda$ and $\phi'$ agree on $X_C$ it follows that there is an $N \in \mathbb Z$ such that $\lambda(x)_{(-\infty,N]} = \phi'(x)_{(-\infty,N]}$. Since $L_{H'}$ is right-resolving and $L_{H'}(\lambda(x)) = \psi(L_{G'}(x)) = L_{H'}(\phi'(x))$, we find that $\lambda(x) = \phi'(x)$. Since such $x$'s are dense in $X_{G'}$ we conclude that $\lambda = \phi'$. 
\end{proof}

We note that in analogy with \eqref{24-11-24x} in Krieger's theorem, the diagram
\begin{equation}\label{30-10-25d}
\xymatrix{
 X_{G'} \ar[r]^-{\phi'} &  X_{H'}  \\
Y \ar[u]^-{\beta_G}  \ar[r]^-\psi & Z \ar[u]_-{\beta_H}}
\end{equation}
commutes.

For a sofic shift $Y$, set
$$
(\mathbb K'(Y),L_{\mathbb K'(Y)})  := (\mathbb K(Y)', L_{\mathbb K(Y)'}).
$$
By Theorem \ref{Krieger} and Theorem \ref{31-10-25hx} this defines a strongly canonical cover for sofic shifts:

\begin{thm}\label{20-01-26m} Let $\psi : Y \to Z$ be a conjugacy of sofic shifts. There is a unique conjugacy $\psi_{\mathbb K'} : X_{\mathbb K'(Y)} \to X_{\mathbb K'(Z)}$ such that
\begin{equation*}
\xymatrix{
 X_{\mathbb K'(Y)}  \ar[r]^-{\psi_{\mathbb K'}} \ar[d]_{L_{\mathbb K'(Y)}}&  X_{\mathbb K'(Z)} \ar[d]^{L_{\mathbb K'(Z)}} \\
Y   \ar[r]^-\psi & Z}
\end{equation*}
commutes.
\end{thm}

For an irreducible sofic shift $Y$ we set
$$
(\mathbb F'(Y),L_{\mathbb F'(Y)})  := (\mathbb F(Y)', L_{\mathbb F(Y)'}).
$$
By Corollary \ref{30-09-25bx} and Theorem \ref{31-10-25hx} this defines also a strongly canonical cover for irreducible sofic shifts:

\begin{thm}\label{20-01-26n} Let $\psi : Y \to Z$ be a conjugacy of irreducible sofic shifts. There is a unique conjugacy $\psi_{\mathbb F'} : X_{\mathbb F'(Y)} \to X_{\mathbb F'(Z)}$ such that
\begin{equation*}
\xymatrix{
 X_{\mathbb F'(Y)}  \ar[r]^-{\psi_{\mathbb F'}} \ar[d]_{L_{\mathbb F'(Y)}}&  X_{\mathbb F'(Z)} \ar[d]^{L_{\mathbb F'(Z)}} \\
Y   \ar[r]^-\psi & Z}
\end{equation*}
commutes.
\end{thm}

\begin{ex}\label{evenshift}
It seems to be a well-established tradition that examples of sofic shifts start with the even shift, and we have no reason to deviate from this. The minimal right-resolving presentation $(\mathbb F(E),L_{\mathbb F(E)})$ of the even shift $E$ is shown in Figure \ref{03-10-25e} and the future cover $(\mathbb K(E),L_{\mathbb K(E)})$ in Figure \ref{11-09-25cx}. Then $(\mathbb F '(E),L_{\mathbb F'(E)})$ is the following labeled graph,

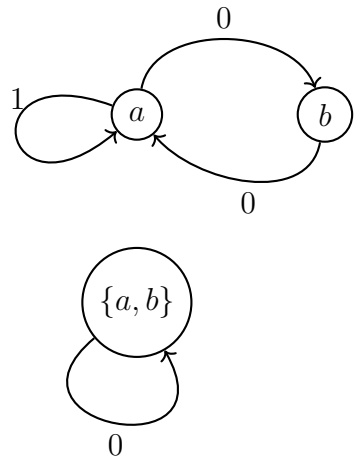
\begin{equation*}
\begin{tikzpicture}[node distance={25mm}, thick, main/.style = {draw, circle}] 
\node[main] (1) {$a$}; 
\node[main] (2) [right of=1] {$b$}; 
\node[main] (3) [below of=1] {$\{a,b\}$};
\draw[->] (1) to [out=80,in=110,looseness=1.0] node[above ,pos=0.5] {$0$} (2); 
\draw[->] (2) to [out=260,in=310,looseness=1.0]  node[below ,pos=0.5] {$0$}(1); 
\draw[->] (1) to [out=160,in=220,looseness=15] node[above ,pos=0.4] {$1$} (1);
\draw[->] (3) to [out=220,in=300,looseness=5] node[below ,pos=0.5] {$0$} (3);
\end{tikzpicture} 
\end{equation*}

and $(\mathbb K'(E), L_{\mathbb K'(E)})$ is the labeled graph 

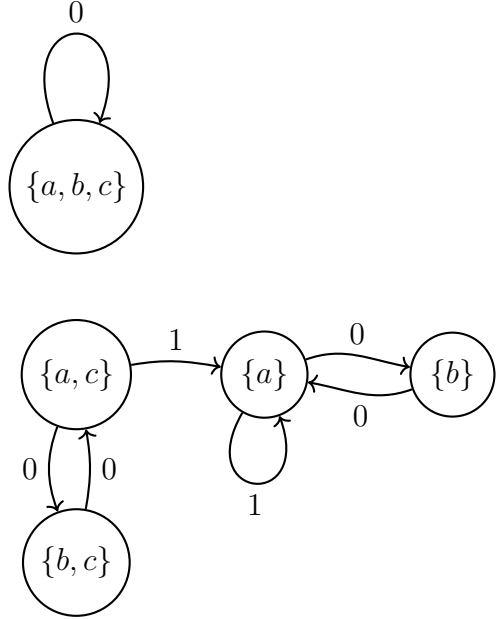
\begin{equation*}
\begin{tikzpicture}[node distance={25mm}, thick, main/.style = {draw, circle}] 
\node[main] (1) {$\{a,b,c\}$}; 
\node[main] (2) [below of=1] {$\{a,c\}$}; 
\node[main] (3) [right of=2] {$\{a\}$};
\node[main] (4) [right of=3] {$\{b\}$};
\node[main] (5) [below of=2] {$\{b,c\}$};
\draw[->] (1) to [out=110,in=70,looseness=7] node[above ,pos=0.5] {$0$} (1);
\draw[->] (2) to [out=250,in=110,looseness=1] node[left, pos=0.5] {$0$} (5);
\draw[->] (5) to [out=80,in=280,looseness=1] node[right, pos=0.5] {$0$} (2);
\draw[->] (2) to [out=10,in=170,looseness=1] node[above, pos=0.5] {$1$} (3);
\draw[->] (3) to [out=240,in=290,looseness=7] node[below, pos=0.5] {$1$} (3);
\draw[->] (3) to [out=20,in=170,looseness=1] node[above, pos=0.5] {$0$} (4);
\draw[->] (4) to [out=200,in=350,looseness=1] node[below, pos=0.5] {$0$} (3);
\end{tikzpicture} 
\end{equation*}

We see therefore, already from the even shift, that in general the canonical covers $(\mathbb F(E),L_{\mathbb F(E)})$, $(\mathbb K(E),L_{\mathbb K(E)})$, $(\mathbb F'(E),L_{\mathbb F'(E)})$ and $(\mathbb K'(E),L_{\mathbb K'(E)})$ are all different.
\end{ex}



\begin{thebibliography}{WWWWW} 



\bibitem[B]{B} M. Barnsley, {\em Fractals everywhere}, Academic Press, 1988.

\bibitem[BKM]{BKM} M. Boyle, B. Kitchens and B.H. Marcus, {\em A note on minimal covers for sofic systems}, Proc. Amer. Math. Soc. {\bf 95} (1985), 403-411.


\bibitem[F]{F} R. Fischer, {\em Sofic systems and graphs}, Monatsh. Math. {\bf 80} (1975), 179--186.


\bibitem[J]{J} R. Johansen, {\em On the structure of covers of sofic shifts}, Documenta Mathematica {\bf 16} (2011), 111-131. 

\bibitem[Jo]{Jo} N. Jonoska, {\em Sofic shifts with synchronizing presentations}, Theoret. Comput. Sci. {\bf 158} (1996), 81--115.






\bibitem[Kr1]{Kr1} W. Krieger, {\em On sofic systems I}, Israel J. Math. {\bf 48} (1984), 305-330.

\bibitem[Kr2]{Kr2} W. Krieger, {\em On sofic systems II}, Israel J. Math. {\bf 60} (1987), 167-176.


\bibitem[LM]{LM} D. Lind and B. Marcus, {\em An introduction to symbolic dynamics and coding}, Cambridge University Press, 1995.




\bibitem[MMTW]{MMTW} B. Marcus, T. Meyerovitch, K. Thomsen and C. Wu, {\em
Factorizable embeddings and the period of an irreducible sofic shift}, arXiv:2508.02554 


\bibitem[N]{N} M. Nasu, {\em Topological conjugacy for sofic systems}, Ergod. Th. \& Dynam. Sys. {\bf 6} (1986), 265-280.



\end{thebibliography}
\end{document}